\documentclass[11pt,leqno]{amsart}

\usepackage{amsmath,amsthm,amscd,amssymb,amsfonts, amsbsy}
\usepackage{latexsym}
\usepackage{txfonts}
\usepackage{exscale}
\usepackage{mathrsfs}
\usepackage{esint} 
\usepackage{enumitem} 
\usepackage{color}

\usepackage[
  hmarginratio={1:1},     
  vmarginratio={1:1},     
  textwidth=400pt,			  
	textheight=580pt,        
]{geometry}

\usepackage[colorlinks,citecolor=blue,pagebackref,hypertexnames=false,pagebackref]{hyperref}

\usepackage[utf8]{inputenc}

\usepackage{hyperref}

\allowdisplaybreaks

\numberwithin{equation}{section}

\theoremstyle{plain}
\newtheorem{theorem}[equation]{Theorem}

\newtheorem{corollary}[equation]{Corollary}
\newtheorem{lemma}[equation]{Lemma}

\newtheorem*{Mtheorem}{Main Theorem}

\theoremstyle{definition}
\newtheorem{defn}[equation]{Definition}

\theoremstyle{remark}
\newtheorem{remark}[equation]{Remark}

\numberwithin{equation}{section}

\numberwithin{equation}{section}

\newcommand{\RR}{{\mathbb{R}}}

\newcommand{\C}{\mathcal{C}}
\newcommand{\HH}{\mathcal{H}}

\newcommand{\dist}{\operatorname{dist}}

\newcommand{\R}{\mathcal{R}}

\newcommand{\pom}{\partial\Omega}

\renewcommand{\emptyset}{\mbox{\textup{\O}}}

\DeclareMathOperator*{\osc}{osc}

\DeclareMathOperator{\diam}{diam}
\DeclareMathOperator{\interior}{int}

\DeclareMathOperator*{\Lip}{Lip}

\def\XXint#1#2#3{{\setbox0=\hbox{$#1{#2#3}{\int}$}
     \vcenter{\hbox{$#2#3$}}\kern-.5\wd0}}

\DeclareMathOperator{\divg}{div}
\DeclareMathOperator{\spt}{spt}

\DeclareMathOperator{\capacity}{cap_2}

\renewcommand{\R}{\mathbb{R}}
\newcommand{\NN}{\mathbb{N}}

\newcommand{\pO}{\partial\Omega}

\newcommand{\Oj}{\Omega_j}
\newcommand{\Oinf}{\Omega_\infty}
\newcommand{\pOj}{\partial\Omega_j}
\newcommand{\pOinf}{\partial\Omega_\infty}

\newcommand{\sj}{\sigma_j}
\newcommand{\sinf}{\mu_\infty}
\newcommand{\oj}{\omega_j}
\newcommand{\oinf}{\omega_{L_\infty}}
\newcommand{\uj}{u_j}

\newcommand{\wXj}{X_j}
\newcommand{\wLj}{L_j}
\newcommand{\wAj}{A_j}
\newcommand{\wcalA}{\mathcal{A}}

\newcommand{\wdj}{\delta_j}

\newcommand{\mj}{\mu_j}
\newcommand{\minf}{\mu_\infty}

\newcommand{\ojA}{\oj^{A(p,r)}}

\reversemarginpar

\begin{document}
\allowdisplaybreaks

\title[Uniform rectifiability and elliptic operators. Part I.]{Uniform rectifiability and elliptic operators satisfying a Carleson measure 
	condition. Part I: The small constant case}

\author[S. Hofmann]{Steve Hofmann}

\address{Steve Hofmann
\\
Department of Mathematics
\\
University of Missouri
\\
Columbia, MO 65211, USA} \email{hofmanns@missouri.edu}

\author[J.M. Martell]{José María Martell}

\address{José María Martell
	\\
	Instituto de Ciencias Matemáticas CSIC-UAM-UC3M-UCM
	\\
	Consejo Superior de Investigaciones Científicas
	\\
	C/ Nicolás Cabrera, 13-15
	\\
	E-28049 Madrid, Spain} \email{chema.martell@icmat.es}

\author[S. Mayboroda]{Svitlana Mayboroda}

\address{Svitlana Mayboroda
\\
Department of Mathematics
\\
University of Minnesota
\\
Minneapolis, MN 55455, USA} \email{svitlana@math.umn.edu}

\author[T. Toro]{Tatiana Toro}

\address{Tatiana Toro 
\\ 
University of Washington 
\\
Department of Mathematics 
\\
Seattle, WA 98195-4350, USA}

\email{toro@uw.edu}

\author[Z. Zhao]{Zihui Zhao}

\address{Zihui Zhao
	\\ 
	Department of Mathematics
	\\
	University of Chicago
	\\
	Chicago, IL 60637, USA}

\email{zhaozh@uchicago.edu}

%
%

\thanks{The first author was partially supported by NSF grant number DMS-1664047.
The second author acknowledges that
the research leading to these results has received funding from the European Research
Council under the European Union's Seventh Framework Programme (FP7/2007-2013)/ ERC
agreement no. 615112 HAPDEGMT. He also acknowledges financial support from the Spanish Ministry of Economy and Competitiveness, through the ``Severo Ochoa Programme for Centres of Excellence in R\&D'' (SEV-2015-0554). 
The third author was partially supported by the NSF INSPIRE Award DMS 1344235, the NSF RAISE-TAQ grant DMS 1839077, and the Simons 
Foundation grant 563916, SM. 
The fourth author was partially supported by the Craig McKibben \& Sarah Merner Professor in Mathematics, by NSF grant number DMS-1664867, and by the Simons Foundation Fellowship 614610.
The fifth author was partially supported by NSF grants DMS-1361823, DMS-1500098, DMS-1664867, DMS-1902756 and by the Institute for Advanced Study.}
\thanks{This material is based upon work supported by the National Science Foundation under Grant No. DMS-1440140 while the authors were in residence at the Mathematical Sciences Research Institute in Berkeley, California, during the Spring 2017 semester.}

\date{\today}


\begin{abstract}
The present paper, along with its sequel, establishes the correspondence between the properties of the solutions of a class of PDEs and the geometry of sets in Euclidean space. We settle the question of whether (quantitative) absolute continuity of the elliptic measure with respect to the surface measure and uniform rectifiability of the boundary are equivalent, in an optimal class of divergence form elliptic operators satisfying a suitable Carleson measure condition. The result can be viewed as a quantitative analogue of the Wiener criterion adapted to the singular $L^p$ data case.

This paper addresses the free boundary problem under the assumption of smallness of the Carleson measure of the coefficients. Part II of this work develops an extrapolation argument to bootstrap this result to the general case. The ideas in Part I constitute a novel application of techniques developed in geometric measure theory. They highlight the synergy between several areas. The ideas developed in this paper are well suited to study singularities arising in variational problems in a geometric setting.
\end{abstract}

\maketitle

\tableofcontents

\section{Introduction}

The present paper, together with its companion  \cite{HMMTZ} and its converse in \cite{KP} (see also \cite{DJe}) culminate many years of activity at the intersection of harmonic analysis, geometric measure theory, and PDEs, devoted to the complete understanding of necessary and sufficient conditions on the operator and the geometry of the domain guaranteeing absolute continuity of the elliptic measure with respect to the surface measure of the boundary.

The celebrated 1924 Wiener criterion \cite{Wiener} provided the necessary and sufficient conditions on the geometry of the domain responsible for the continuity of the harmonic functions at the boundary. In the probabilistic terms, it characterized the points of the boundary which are  ``seen" by the Brownian travelers coming from the interior of the domain. 

The question of finding necessary and sufficient geometric conditions which could guarantee adequate regularity, so that, roughly speaking, the pieces of the boundary are seen by the Brownian travelers according to their surface measure, turned out to be much more intricate. Curiously, already in 1916 F. \& M. Riesz correctly identified the key geometric notion in this context:  rectifiability of the boundary $\pom$, i.e., the existence of tangent planes almost everywhere 
with respect to arc length $\sigma$ on $\pom$.  In particular, they showed in  \cite{Rfm} that harmonic measure is (mutually) absolutely continuous with respect to $\sigma$ for
a simply connected domain in the plane with rectifiable boundary. It took more than a hundred years to establish the converse of the F. \& M. Riesz theorem and its higher dimensional analogues. The first such result appeared in 2016 \cite{7au}, and the question was fully settled for the harmonic functions in 2018 \cite{AHMMT}.

 The question of what happens in the general PDE setting has been puzzling from the beginning.
The Wiener criterion is universal: it applies to all uniformly elliptic divergence form operators with bounded coefficients and characterizes
 points of continuity of the solution at the boundary. It was realized early 
on that no such general criterion exists for determining the absolute continuity of elliptic measure with respect to the surface measure to the boundary of a domain.
Some of the challenges that arise when considering this question were highlighted by the counterexamples in \cite{CFK}, \cite{MM}. In 1984 Dahlberg 
formulated a conjecture concerning optimal conditions on a matrix of coefficients which guarantee absolute continuity of elliptic measure with respect to 
Lebesgue measure in a half-space. This question was a driving force of a thread of outstanding developments in harmonic analysis in the 80s and 90s  due to Dahlberg, Jerison, Kenig, Pipher, and others, stimulating some beautiful and far-reaching new techniques in the theory of weights and singular integral operators, to mention only a few approaches.  In \cite{KP} Kenig and Pipher proved Dahlberg's conjecture, they showed that whenever the gradient of coefficients satisfies a Carleson measure condition (to be defined below) the elliptic measure and the Lebesgue measure are mutually absolutely  continuous on a half-space and, by a change of variables argument, above a  Lipschitz graph. 


The main goal of the present pair of papers is to establish the equivalence of the absolute continuity of the elliptic measure with respect to the surface measure and the uniform rectifiability of the boundary of a domain under the Dahlberg-Kenig-Pipher (DKP) condition on the coefficients, thus providing the final, optimal results geometrically (given the assumed background 
hypotheses) and in terms of the operator at hand.


It was natural to conjecture that the equivalence of rectifiability and regularity of elliptic measure should be valid in the full generality of DKP coefficients.
Despite numerous attempts this question turned out to be notoriously resistant to existing methods.
The passage from the regularity of the solutions to partial differential equations to rectifiability, or to any geometric information on the boundary, is generally referred to as free boundary problems. This in itself is, of course, a well-studied and rich subject. Unfortunately, the typical techniques arising from minimization of the functionals are both too qualitative and too rigid to treat structural irregularities of  rectifiable sets and such weak assumptions as absolute continuity of harmonic measure. The latter became accessible only recently, with the development of the analysis of singular integrals and similar objects on uniformly rectifiable sets. In particular, the first converse of the F. \& M. Riesz theorem, \cite{7au}, directly relies on the 2012 solution of the David-Semmes conjecture regarding the boundedness of the Riesz transforms in $L^2$ \cite{NTV}. At the same time, the techniques stemming from such results for the harmonic functions are not amenable to more general operators of the DKP type, again, due to simple yet fundamental algebraic deficiencies: the derivatives of the coefficients do not offer sufficient cancellations. The present paper pioneers a different route, developing a third approach. It combines, or rather interlaces, the ``classical" free boundary blow-up and compactness arguments (originated in geometric measure theory) with the scale-invariant harmonic analysis methods. This allows us to take advantage of the appropriate amelioration of the coefficients 
obtained via a compactness approach.
Subsequently, in the second paper, where we address the conjecture in full generality, we manipulate and, ultimately, bootstrap (``extrapolate")  the results from this paper in an intricate stopping-time argument. This 
also requires the development of a transference principle.

Let us now discuss these issues at a more technical level. Throughout the paper we shall work under the assumptions that the domain $\Omega$ is uniform, i.e., open and connected in a quantitative way, and that its boundary is $(n-1)$-Ahlfors regular, that is, $(n-1)$-dimensional in a quantitative way (see Section \ref{sPrelim}). Under these conditions one can, for instance,  show that scale-invariant absolute continuity of harmonic  measure is related to the uniform  rectifiability of the boundary and even to the non-tangential accessibility of the exterior domain:

\begin{theorem}\label{thm:hmu}
Let $\Omega\subset\R^n$, $n\ge 3$, be a uniform domain (bounded or unbounded) with Ahlfors regular boundary \textup{(}cf. Definitions \ref{def:uniform} and \ref{def:ADR}\textup{)}, set $\sigma=\mathcal{H}^{n-1}|_{\pO}$, and let $\omega_{-\Delta}$ denote its associated harmonic measure. The following statements are equivalent:

\begin{enumerate}[label=\textup{(\alph*)}, itemsep=0.2cm]

\item\label{1-thm:hmu} $\omega_{-\Delta} \in A_\infty(\sigma)$ \textup{(}cf. Definition \ref{def:AinftyHMU}\textup{)}.

\item\label{2-thm:hmu} $\partial\Omega$ is uniformly rectifiable. \textup{(}cf. Definition \ref{def:UR} \textup{)}.

\item\label{3-thm:hmu} $\Omega$ satisfies the exterior corkscrew condition \textup{(}cf. Definition \ref{def:ICC}), hence, in particular, it is a chord-arc domain \textup{(}cf. Definition \ref{def:nta}).
\end{enumerate}
 \end{theorem}
 
 Postponing all the rigorous definitions to Section~\ref{sPrelim}, we remark for the moment that uniform rectifiability is a quantitative version of the notion of rectifiability of the boundary and the Muckenhoupt condition $\omega\in A_\infty(\sigma)$ is, respectively, a quantitative form of the mutual absolute continuity of $\omega$ with respect to $\sigma$. Thus, Theorem \ref{thm:hmu} above is a quantitative form of the rigorous connection between the boundary behavior of harmonic functions and geometric properties of sets that we alluded to above. 
Returning to the ties with Wiener criterion, we point out that the property of the scale invariant absolute continuity of harmonic measure with respect to surface measure, 
at least in the presence of Ahlfors regularity of $\pom$,
 is equivalent to the solvability of the Dirichlet problem
with data in some $L^p(\pom)$, with $p<\infty$\footnote{See, e.g., \cite{H}, although the result is folkloric,
and well known in less austere settings \cite{Ke}.}; thus, such a characterization is in some sense
an analogue of Wiener's criterion for singular, rather than continuous data.

Theorem~\ref{thm:hmu} in the present form appears in \cite[Theorem 1.2]{AHMNT}. That \ref{1-thm:hmu} implies \ref{2-thm:hmu} is the main result in \cite{HMU} (see also \cite{HM4, HLMN});  that \ref{2-thm:hmu} yields \ref{3-thm:hmu} is \cite[Theorem 1.1]{AHMNT}; and the fact that \ref{3-thm:hmu}  implies \ref{1-thm:hmu} was proved in \cite{DJe}, and independently in \cite{Sem}. 

Theorem \ref{thm:hmu} 
and other recent results\footnote{We refer the reader also to recent work of Azzam \cite{Az}, in 
which the author characterizes the domains with Ahlfors regular boundaries for which $\omega_{-\Delta}\in A_\infty(\sigma)$:  they are precisely the domains with uniformly rectifiable boundary which are semi-uniform in the sense of Aikawa and Hirata \cite{AH};  see also
\cite{AHMMT, AMT-char, HM3} for related results characterizing $L^p$ solvability in the general
case that $\omega_{-\Delta}$ need not be doubling.} illuminate how the $A_\infty$ condition\footnote{And also its
non-doubling version, the weak $A_\infty$ condition.} 
 of  harmonic measure is related to the geometry of the domain $\Omega$. Unfortunately, as we pointed out above, their proofs do not extend to the optimal class of operators with variable coefficients. Indeed, the best known results in this direction  pertain to the ``direct" rather than the ``free boundary" problem. A description of the elliptic measure in a given geometric environment, is essentially due to C. Kenig and  J. Pipher. 
 In 2001  \cite{KP} C. Kenig and  J. Pipher proved what they referred to as a 1984 Dahlberg conjecture: if $\Omega\subset \RR^n$ is a bounded Lipschitz domain and the 
 elliptic matrix $\mathcal{A}$ satisfies the following Carleson measure condition:
\begin{equation}\label{KP-cond}
\sup_{\substack{q\in\pO \\ 0< r<\diam(\Omega)} } \frac{1}{r^{n-1}} \int_{B(q,r) \cap\Omega} 
\bigg(
\sup_{Y\in B(X,\frac{\delta(X)}{2})} |\nabla \mathcal{A}(Y)|^2 \delta(Y)\bigg)dX <\infty,
\end{equation}
where here and elsewhere we write $\delta(\cdot)=\dist(\cdot,\partial\Omega)$, then the 
corresponding elliptic measure $\omega_L\in A_{\infty}(\sigma) $.  
As observed in \cite{HMT1}, 	
one may carry through the proof in \cite{KP}, essentially 
unchanged, with a slightly weakened reformulation of  \eqref{KP-cond}, namely
by assuming, in place of  \eqref{KP-cond}, the following properties: 
\begin{enumerate}[label=(H\arabic*), itemsep=0.2cm]
	\item\label{H1} $\wcalA\in \Lip_{\rm loc}(\Omega)$ and $|\nabla \wcalA|\delta(\cdot) \in L^\infty(\Omega)$, where $\delta(\cdot) := \dist(\cdot,\pO)$.
	\item\label{H2} $|\nabla \wcalA|^2 \delta(\cdot)$ satisfies the Carleson measure assumption:
	\begin{equation}\label{KP-s-relaxed}
	\|\mathcal{A}\|_{\rm Car}:=\sup_{\substack{q\in\pO\\0< r<\diam(\Omega)} } \frac{1}{r^{n-1}} \int_{B(q,r) \cap\Omega} 
	|\nabla \mathcal{A}(X)|^2 \delta(X)dX <\infty\,.
	\end{equation}	
	\end{enumerate}
We shall refer to these hypotheses (jointly) as the \textbf{Dahlberg-Kenig-Pipher  (DKP)  condition}.
Note that each of \ref{H1} and \ref{H2} is implied by \eqref{KP-cond}.

Since properties
\ref{H1} and \ref{H2} are preserved in subdomains, one can use
the method of \cite{DJe} to extend the result of \cite{KP} to chord-arc domains, 
and hence the analogue of \ref{3-thm:hmu} implies \ref{1-thm:hmu} (in Theorem \ref{thm:hmu}) 
holds for operators satisfying the DKP condition. 
 \emph{This is an optimal result in the sense that weaker regularity of the coefficients could give rise to the elliptic measure which is singular with respect to the surface measure even for the half-plane \cite{CFK, MM, HMMTZ}.}

The ``free boundary" part of the problem, however, turned out to be extremely challenging.  In an effort to prove that  \ref{1-thm:hmu} implies  \ref{2-thm:hmu} or  \ref{3-thm:hmu} 
for this class of operators, the first, second and fourth authors of the present paper have recently obtained in \cite{HMT1} that under the same background hypothesis as in Theorem \ref{thm:hmu},  \ref{1-thm:hmu} implies 
 \ref{3-thm:hmu} (and hence also  \ref{2-thm:hmu})
 for elliptic operators with variable-coefficient matrices $\mathcal{A}$ 
 satisfying \ref{H1} and the Carleson measure estimate 
\begin{equation}\label{HMT-CM}
\sup_{\substack{q\in\pO \\ 0< r<\diam(\Omega)} } \frac{1}{r^{n-1}} \int_{B(q,r) \cap\Omega} 
|\nabla \mathcal{A}(X)|dX <\infty.
\end{equation}
We observe that, in the presence of hypothesis \ref{H1}, 
the latter condition is stronger than the relaxed DKP condition \eqref{KP-s-relaxed}; 
hence Theorem  \ref{thm:hmu} remains true for this new class of matrices. We also mention in
 passing that a qualitative version of this fact was obtained in \cite{ABHM}, 
and that there are
related (quantitative) results in \cite{HMM} and \cite{AGMT} that are 
valid in the absence of any connectivity hypothesis.

Nonetheless, it has remained an open problem to address such issues, assuming only
that the coefficients satisfy the DKP condition (specifically, the weighted
$W^{1,2}$ Carleson measure
estimate \eqref{KP-s-relaxed}, as opposed to the $W^{1,1}$ version \eqref{HMT-CM}).
Observe that the former is both weaker, and more natural:  for example, operators verifying 
 \eqref{KP-s-relaxed} arise as pullbacks of constant coefficient operators (see \cite[Introduction]{KP}), and also
 in the linearization of ``$A$-harmonic" (i.e., generalized $p$-harmonic) operators (see \cite[Section 4]{LV}). In particular, in light of \cite{HMMTZ} and the 
 examples included there the DKP condition is sharp.

All in all, the main motivation for this paper and its companion \cite{HMMTZ} is to understand whether the  elliptic measure of a DKP divergence form elliptic 
operator distinguishes between a rectifiable and a purely unrectifiable boundary. 
As in Theorem \ref{thm:hmu}, we make the background assumption
that $\Omega\subset \mathbb{R}^n$, $n\ge 3$, is a uniform domain
(see Definition \ref{def:uniform}) 
 with an Ahlfors regular boundary (Definition 
\ref{def:ADR}).
Analytically we consider second order divergence form symmetric elliptic
operators, that is, 
$L=-\divg(\mathcal{A}(\cdot)\nabla)$, where $\mathcal{A}= \big( a_{ij}\big)_{i,j=1}^n$ is a  
real symmetric matrix-valued function on $\Omega$, 
satisfying the usual uniform ellipticity condition 
\begin{equation}\label{def:UE}
	\lambda |\xi|^2 \leq \langle \mathcal{A}(X)\xi, \xi \rangle \leq \Lambda |\xi|^2, \qquad\text{for all } \xi\in\mathbb{R}^n\setminus\{0\}\,,
\end{equation}
for uniform constants $0<\lambda\le \Lambda<\infty$, and for a.e. $X\in\Omega$.
We further assume that $\mathcal{A}$ 
satisfies the DKP condition, that is, \ref{H1} and \ref{H2}, and, additionally, that the associated 
elliptic measure is an $A_\infty$ weight (see Definition \ref{def:AinftyHMU}) with 
respect to the surface measure $\sigma=\mathcal{H}^{n-1}|_{\pO}$. Our goal is to understand how this 
analytic information yields geometric insight regarding the geometry of the
domain and its boundary. 
We do this in two stages. First, we consider the 
case in which the quantity \eqref{KP-s-relaxed} is sufficiently small. 
This is the content of the present paper, and in this small constant setting, we can go even further, 
assuming a much weaker condition than \eqref{KP-s-relaxed} (see \eqref{def:smallCarleson} for the 
precise statement of the
assumption, and Remark \ref{remark:mainthm}). The second stage is 
taken in \cite{HMMTZ}, the sequel to the present paper, 
where we derive the ``large constant'' case (even for non-symmetric operators) by taking our small constant result as a starting point, 
and then using an extrapolation (i.e., bootstrapping) argument 
to pass to the case in which the constant
in \eqref{KP-s-relaxed} is assumed merely to be finite. 
We remark that it is at this stage that we use the full strength of the Carleson condition
\eqref{KP-s-relaxed}.

Let us make some preliminary
observations and introduce some notation. First, we  shall always work in $\R^n$, $n\ge 3$.
Throughout this paper, and unless otherwise specified, by \textit{allowable constants}, 
we mean the dimension $n\geq 3$; the constants involved in the definition of a uniform 
domain, that is, $M, C_1>1$ (see Definition \ref{def:uniform}); the Ahlfors regular constant 
$C_{AR}>1$ (see Definition \ref{def:ADR}); the ratio of the ellipticity constants 
$\Lambda/\lambda \geq 1$ (see \eqref{def:UE}), 
and the $A_{\infty}$ constants $C_0>1$ and 
$\theta\in(0,1)$ (see Definition 
\ref{def:AinftyHMU}).  
By renormalizing, we shall always assume that $\lambda = 1$; thus the new, renormalized
$\Lambda$ will be equal to the ratio $\Lambda/\lambda$ of the original ellipticity constants. 
See Remark \ref{normalize} below.

We are now ready to state our main result.

\begin{Mtheorem}\label{thm:main}
Given the values of allowable constants $n\ge 3$, $M, C_1, C_{AR}>1$, $\Lambda \ge \lambda =1$, 
$C_0>1$, and $0<\theta<1$, there exist $N$ and $\epsilon >0$ depending on the allowable constants, such that the following holds.  Let $\Omega\subset\mathbb{R}^n$ be a bounded uniform domain with constants $M, C_1$ and whose boundary $\pO$ is Ahlfors regular with constant $C_{AR}$ and set $\sigma=\mathcal{H}^{n-1}|_{\pO}$. Let $L = -\divg(\mathcal{A}(\cdot)\nabla)$ be an elliptic operator with real symmetric matrix $\mathcal{A}$ satisfying \eqref{def:UE} with ellipticity constants $1=\lambda\leq \Lambda$ such that the corresponding elliptic measure satisfies $\omega_{L} \in A_{\infty}(\sigma)$ with constants $C_0$ and $\theta$. If $\mathcal{A}$ verifies
	    \begin{equation}\label{def:smallCarleson}
	    		\C(\Omega,\mathcal{A}) : = \sup_{X\in\Omega}  \fint_{B\left(X,\delta(X)/2 \right)} |\nabla \mathcal{A}(Y)|\delta(Y)  dY  < \epsilon,
	    \end{equation}
			where $\delta(\cdot)=\dist(\cdot,\partial\Omega)$, then $\Omega$ satisfies the exterior corkscrew condition with constant $N$.
\end{Mtheorem}

\begin{corollary}
	Under the same assumption as the Main Theorem, $\Omega$ has uniformly rectifiable boundary.
\end{corollary}

\begin{remark}\label{remark:mainthm}
\ \vskip-.8cm\  

\begin{enumerate}[label=\textup{(\roman*)}, itemsep=0.2cm]

\item We note that our assumption \eqref{def:smallCarleson} on the matrix $\mathcal{A}$ is much weaker than the smallness of the relaxed DKP condition \eqref{KP-s-relaxed}. To see this, given $X\in\Omega$, let $q_X \in\pO$ be such that $|X-q_X| = \delta(X)$. Then by Hölder's inequality
\begin{equation}\label{holder-cond}
\fint_{B\left(X,\delta(X)/2\right)} |\nabla \mathcal{A}(Y)|\delta(Y)dY
\lesssim 
\left(\frac{1}{\delta(X)^{n-1}} \int_{B\left(q_X, 3\delta(X)/2 \right) \cap\Omega} |\nabla \mathcal{A}(Y)|^2 \delta(Y) dY\right)^{\frac12} 
.\end{equation}
Hence \eqref{KP-s-relaxed} with sufficiently small constant gives \eqref{def:smallCarleson}. On the other hand, it is easy to see that the latter is much weaker. Assume for instance that $|\nabla \mathcal{A}|\delta\sim \epsilon$ in $\Omega$ in which case  \eqref{def:smallCarleson} holds but
\eqref{KP-s-relaxed} fails since every integral is infinity.

%

\item As will be pointed out in the proof (see Remark \ref{rem:oscila-KP}) our condition \eqref{def:smallCarleson} can be relaxed by assuming that 
	  		\begin{equation}\label{def:oscA}
	  			\osc(\Omega, \mathcal{A}) := \sup_{X\in\Omega} \fint_{B(X,\delta(X)/2)} |\mathcal{A}(Y) - \langle\mathcal{A}\rangle_{B(X,\delta(X)/2)}| dY <\epsilon,
	  		\end{equation}
where $\langle\wcalA\rangle_{B(X,\delta(X)/2)}$ denotes the average of $\wcalA$ on $B(X,\delta(X)/2)$. 

\item The hypothesis of Main Theorem, boundedness of the domain and symmetry of the operator might seem restrictive. Nevertheless this result is enough
to prove the general case for operators (not necessarily symmetric) satisfying  \eqref{def:smallCarleson} on any uniform domain with Ahlfors regular boundary.

\end{enumerate}    
\end{remark}

{
\begin{remark}\label{normalize}
We note that the $A_\infty$ constants for $\omega_L$ are not affected by the normalization $\lambda = 1$, however,
the small parameter $\epsilon$ in \eqref{def:smallCarleson} clearly depends upon this normalization.
\end{remark}
}

\begin{remark}\label{rem:subtle}
Having fixed the desired 
ellipticity constants $\lambda=1$ and $\Lambda$ and the geometric parameters $M, C_1, C_{AR}>1$, one may ask whether operators $L= -\divg(\mathcal{A}(\cdot)\nabla)$ such that $\omega_L\in A_{\infty}(\sigma)$ and 
 $\wcalA$ satisfies \eqref{def:smallCarleson} with 
small constant $\epsilon$ exist. Choosing a matrix for which the left-hand side of \eqref{KP-cond} is small (e.g., a constant coefficient matrix), we can guarantee that \eqref{def:smallCarleson} holds with a desired $\epsilon$, see \eqref{holder-cond}. It is a consequence of the work in \cite{KP} that on a chord arc domain (see Definition \ref{def:nta})
the $A_{\infty}$ constants of $\omega_L$ only depend on the ellipticity constants, the norm \eqref{KP-cond} and the geometric parameters (which include
$M, C_1, C_{AR}>1$). Thus in this case there exist constants $C_0>1$ and $\theta\in (0,1)$ such that all the conditions of the Main Theorem are satisfied.

\end{remark}

\section{Preliminaries}\label{sPrelim}



\begin{defn}\label{def:ADR}
	We say a closed set $E\subset \RR^n$ is \textbf{Ahlfors regular} with constant $C_{AR}>1$ if for any $q\in E$ and $0<r<\diam(E)$,
	\[ C_{AR}^{-1}\, r^{n-1} \leq \mathcal{H}^{n-1}(B(q,r)\cap E) \leq C_{AR}\, r^{n-1}. \]
\end{defn}

There are many equivalent characterizations of a uniformly rectifiable set, see \cite{DS2}. Since uniformly rectifiability is not the main focus of our paper, we only state one of the geometric characterizations as its definition. 
\begin{defn}\label{def:UR}
	An Ahlfors regular set $E\subset \RR^n$ is said to be \textbf{uniformly rectifiable}, if it has big pieces of Lipschitz images of $\RR^{n-1}$. That is, there exist $\theta, M>0$ such that for each $q\in E$ and $0<r<\diam(E)$, there is a Lipschitz mapping $\rho: B_{n-1}(0, r) \to \RR^n$ such that $\rho$ has Lipschitz norm $\leq M$ and 
	\[ \mathcal{H}^{n-1} \left( E\cap B(q,r) \cap \rho(B_{n-1}(0,r)) \right) \geq \theta r^{n-1}. \]
	Here $B_{n-1}(0,r)$ denote a ball of radius $r$ in $\RR^{n-1}$.
\end{defn}

\begin{defn}\label{def:ICC}
An open set $\Omega\subset\mathbb{R}^n$ is said to satisfy the \textbf{\textup{(}interior\textup{)} corkscrew condition} \textup{(}resp. the exterior corkscrew condition\textup{)} with constant $M>1$ if for every $q\in\pO$ and every $0< r<\diam(\Omega)$, there exists $A=A(q,r) \in \Omega$ \textup{(}resp. $A\in \Omega_{\rm ext}:=\mathbb{R}^n\setminus\overline{\Omega}$\textup{)} such that
 \begin{equation}\label{eqn:nta-M}
 	B\left(A, \frac{r}{M} \right) \subset B(q,r) \cap \Omega
	\qquad
	\Big(\mbox{resp. }B\left(A, \frac{r}{M} \right) \subset B(q,r) \cap \Omega_{\rm ext}.
	\Big)
	 \end{equation} 
	 The point $A$ is called  a Corkscrew point (or a non-tangential point)  relative to $\Delta(q,r)=B(q,r)\cap\pom$ in $\Omega$ (resp. $\Omega_{\rm ext}$).
\end{defn}

\begin{defn}\label{def:HCC}
An open connected set $\Omega\subset\mathbb{R}^n$ is said to satisfy the \textbf{Harnack chain condition} with constants $M, C_1>1$ if for every pair of points $A, A'\in \Omega$
there is a chain of balls $B_1, B_2, \dots, B_K\subset \Omega$ with $K \leq  M(2+\log_2^+ \Pi)$ that connects $A$ to $A'$,
where
\begin{equation}\label{cond:Lambda}
\Pi:=\frac{|A-A'|}{\min\{\delta(A), \delta(A')\}}.
\end{equation} 
Namely, $A\in B_1$, $A'\in B_K$, $B_k\cap B_{k+1}\neq\emptyset$ and for every $1\le k\le K$
\begin{equation}\label{preHarnackball}
 	C_1^{-1} \diam(B_k) \leq \dist(B_k,\partial\Omega) \leq C_1 \diam(B_k).
\end{equation}
         \end{defn}

We note that in the context of the previous definition if $\Pi\le 1$ we can trivially form the Harnack chain $B_1=B(A,3\delta(A)/5)$ and $B_2=B(A', 3\delta(A')/5)$ where \eqref{preHarnackball} holds with $C_1=3$. Hence the Harnack chain condition is non-trivial only when $\Pi> 1$.

\begin{defn}\label{def:uniform}
An open connected set $\Omega\subset\RR^n$ is said to be a \textbf{uniform} domain with constants $M, C_1$,  if it satisfies the interior corkscrew condition with constant $M$ and the Harnack chain condition with constants $M, C_1$.
\end{defn}

\begin{defn}\label{def:nta}
A uniform domain $\Omega\subset\RR^n$ is said to be \textbf{NTA} if it satisfies the exterior corkscrew condition. If one additionally assumes that $\partial\Omega$ is Ahlfors regular, the $\Omega$ is said to be a \textbf{chord-arc} domain.
\end{defn}

For any $q\in\pO$ and $r>0$, let $\Delta=\Delta(q,r)$ denote the surface ball $B(q,r) \cap \pO$, and let $T(\Delta)=B(q,r)\cap\Omega$ denote the Carleson region above $\Delta$. We always implicitly assume that $0<r< \diam ( \Omega)$. We will also write $\sigma=\mathcal{H}^{n-1}|_{\pO}$.

Given an open connected set $\Omega$ and an elliptic operator $L$ we let $\{\omega_L^X\}_{X\in\Omega}$ be the associated elliptic measure.  In the statement of the Main Theorem we assume that $\omega_L \in A_{\infty}(\sigma)$ in the following sense:

\begin{defn}\label{def:AinftyHMU}
	The elliptic measure associated with $L$ in $\Omega$ is said to be of class $A_{\infty}$ with respect to the surface measure $\sigma= \mathcal{H}^{n-1}|_{\pO}$, which we denote by $\omega_L\in A_\infty(\sigma)$, if there exist $C_0>1$ and $0<\theta<\infty$ such that for any surface ball $\Delta(q,r)=B(q,r)\cap \pO$, with $x\in\partial\Omega$ and $0<r<\diam (\Omega)$, any surface ball $\Delta'=B'\cap \pO$ centered at $\pO$ with $B'\subset B(q,r)$, and any Borel set $F\subset \Delta'$, the elliptic measure with pole at $A(q,r)$ (a corkscrew point relative to $\Delta(q,r)$) satisfies
	   \begin{equation}\label{eqn:Ainfty}
		\frac{\omega_L^{A(q,r)}(F)}{\omega_L^{A(q,r)}(\Delta')} \leq C_0\left( \frac{\sigma(F)}{\sigma(\Delta')} \right)^{\theta}.
	\end{equation}
\end{defn}

\medskip

\begin{defn}\label{def:CDC}
A domain $\Omega\subset \mathbb{R}^n$ with $n\ge 3$ is said to satisfy the \textbf{capacity density condition (CDC)} if there exists a constant $c_0>0$ such that 
\begin{equation}\label{eqn:CDC}
	\frac{\capacity(B_r(q)\setminus\Omega)}{\capacity(B_r(q))} \geq c_0,\quad \text{for any } q\in\pO \text{ and } 0<r<\diam(\Omega),
\end{equation}	
where for any set $K\subset\mathbb{R}^n$, the capacity is defined as  
	\begin{equation*}
		\capacity(K) = \inf \bigg\{\int |\nabla \varphi|^2 dX: \varphi \in C_c^{\infty}(\mathbb{R}^n), K\subset \interior\{\varphi\geq 1\}\bigg\}.
	\end{equation*}
\end{defn}
It was proved in \cite[Section 3 ]{Zh} and \cite[Lemma 3.27]{HLMN} that a domain in $\mathbb{R}^n$, $n\ge 3$, with $(n-1)$-Ahlfors regular boundary satisfies the capacity density condition with constant $c_0$ depending only on $n$ and the Ahlfors regular constant $C_{AR}$. In particular such a domain is Wiener regular and hence for any elliptic operator $L$, and any function $f\in C(\pO)$, we can define
\begin{equation}\label{eqn:ellipt-rep}
	u(X) = \int_{\partial\Omega} f(q) d\omega_L^X(q),
	\qquad X\in\Omega,
\end{equation}
and obtain that $u\in W_{\rm loc}^{1,2}(\Omega)\cap C(\overline\Omega)$, $u|_{\partial\Omega}=f$ on $\pO$ and  $Lu = 0$ in $\Omega$ in the weak sense. Moreover, if additionally $f\in \Lip(\Omega)$ then $u\in W_{}^{1,2}(\Omega)$.

\medskip

Given $L=-\divg(\mathcal{A}\nabla)$ with $\mathcal{A}$ satisfying \eqref{def:UE}, one can construct the associated elliptic measure $\omega_L$ and Green function $G$. For the latter the reader is referred to the work of Grüter and Widman \cite{GW}, while the existence of the corresponding elliptic measures is an application of the Riesz representation theorem. The behavior of $\omega_L$ and $G$, as well as the relationship between them, depends crucially on the properties of $\Omega$, and assuming that $\Omega$ is a uniform domain with the CDC one can follow the program carried out in \cite{NTA}. We summarize below the results which will be used later in this paper. For a comprehensive treatment of the subject we refer the reader to the forthcoming monograph \cite{HMT2}.



\begin{theorem}\label{thm:gw}
Let $L$ be a divergence form elliptic operator in a bounded open connected set $\Omega\subset\RR^n$,
whose coefficient matrix $A$ has been normalized so that $\lambda =1$ in \eqref{def:UE}.
Then there is a unique non-negative function $G: \Omega \times \Omega \to \RR\cup \{\infty\}$, the 
Green function associated with $L$, and a positive, finite constant $C$, depending only on dimension, and 
(given the normalization) $\Lambda$, such that the following hold:
\begin{equation}
G(\cdot, Y) \in W^{1,2}(\Omega \setminus B(Y,s))\cap W^{1,1}_0(\Omega)\cap W_0^{1,r}(\Omega), \quad \forall\,Y\in\Omega,\ \forall\,s>0,\ \forall\,r\in \left[1,\tfrac{n}{n-1}\right);
\end{equation}
\begin{equation}
\int\left\langle A(X)\nabla_X G(X,Y), \nabla \varphi(X) \right\rangle dX = \varphi(Y),
\quad \text{for all }\varphi \in C_c^{\infty}(\Omega);
\end{equation} 
\begin{equation}\label{eqn:qweak}
		\|G(\cdot,Y)\|_{L^{\frac{n}{n-2},\infty}(\Omega)} +\|\nabla G(\cdot,Y)\|_{L^{\frac{n}{n-1},\infty}(\Omega)}\le C, \quad\forall\,Y\in \Omega;
\end{equation}
\begin{equation}\label{eqn:gub}
		G(X,Y) \leq C|X-Y|^{2-n};
	\end{equation}
and
\begin{equation}\label{eqn:glb}
		G(X,Y) \geq C|X-Y|^{2-n}, \quad\text{if\ \ } |X-Y|\leq \frac78{\delta(Y)}.
	\end{equation}	
Furthermore, if $\Omega$ is a uniform domain satisfying the CDC, for any $\varphi \in C_c^{\infty}(\RR^n)$ and for almost all $Y\in\Omega$
	\begin{equation}\label{eqn:int-parts}
		-\int_{\Omega} \left\langle \mathcal{A}(X)\nabla_X G(X,Y), \nabla \varphi(X) \right\rangle dX = \int_{\pO} \varphi d\omega_L^{Y} -\varphi(Y)
	\end{equation}
	where $\{\omega_L^Y\}_{Y\in\Omega}$ is the associated elliptic measure.

\end{theorem}

We observe that \eqref{eqn:qweak} and Kolmogorov's inequality give that for every $1\le r<\frac{n}{n-1}$ 
\begin{equation}\label{eqn:glq}
\|G(\cdot,Y)\|_{L^r(\Omega)}\le C C_3^{\frac1r} 
|\Omega|^{\frac1r-\frac{n-2}n},
\qquad
\|\nabla G(\cdot,Y)\|_{L^r(\Omega)}\le C C_4^{\frac1r}
|\Omega|^{\frac1r-\frac{n-1}n},
	\end{equation}
where $C$ is the constant in \eqref{eqn:qweak}, $C_4=(\frac{n}{(n-2)r})'$, and $C_4=(\frac{n}{(n-1)r})'$.

\begin{lemma}\label{lem:vanishing}
	Let $\Omega$ be a uniform domain satisfying the CDC. There exist constants $C, \beta>0$ (depending on the allowable constants) such that for $q\in\pO$ and $0<r<\diam(\pO)$, and $u\geq 0$ with $Lu=0$ in $ B(q,2r) \cap\Omega$, if $u$ vanishes continuously on $\Delta(q,2r) =  B(q,2r) \cap \pO$, then
	\begin{equation}\label{eqn:1.1}
		u(X) \leq C\left( \frac{|X-q|}{r} \right)^{\beta} \sup_{B(q, 2r)\cap\Omega} u,  \qquad \text{for any} X\in\Omega\cap B(q,r).
	\end{equation}
	\end{lemma}

\begin{lemma}\label{ellip-lb}
	Let $\Omega$ be a uniform domain satisfying the CDC. There exists $m_0\in(0,1)$ depending on the allowable constants such that for any $q\in\pO$ and $0<r<\diam (\pO)$,
	\begin{equation}\label{eqn:1.2}
		\omega_L^{A(q,r)} (\Delta(q,r)) \geq m_0.
	\end{equation}
	Here $A(q,r)$ denotes a non-tangential point for $q$ at radius $r$.
\end{lemma}

\begin{lemma}\label{harn-princ}
	Let $\Omega$ be a uniform domain satisfying the CDC. There exists a constant $C$ (depending on the allowable constants) such that for $q\in\pO$ and $0<r<\diam(\pO)$. If $u\geq 0$ with $Lu=0$ in $\Omega\cap B(q,2r)$ and $u$ vanishes continuously on $\Delta(q,2r)$, then 
	\begin{equation}\label{eqn:1.3}
		u(X) \leq C u(A(q,r)), \qquad \text{for any}\ X\in\Omega\cap B(q,r).
	\end{equation}
	\end{lemma}

\begin{lemma}\label{CFMS}
	Let $\Omega$ be a uniform domain satisfying the CDC. There exists $C>0$ depending on the allowable constants such that for $q\in\pO$ and $0<r<\diam(\pO)/M$,
	\begin{equation}
		C^{-1} \leq \dfrac{\omega_L^X(\Delta(q,r))}{r^{n-2}G(A(q,r),X)} \leq C, \qquad \text{for any } X\in \Omega\setminus B(q,4r).
	\end{equation}
	\end{lemma}

\begin{lemma}\label{doubling}
	Let $\Omega$ be a uniform domain satisfying the CDC. There exists $C>0$ depending on the allowable constants such that for any $q\in\pO$ and $0<r<\diam(\pO)/4$, if $X\in \Omega\setminus B(q,4r)$, then 
	\begin{equation}\label{eqn:1.5}
		\omega_L^X(\Delta(q,2r)) \leq C\omega_L^X(\Delta(q,r)).
	\end{equation}
	\end{lemma}

\medskip

\begin{remark}\label{rem:doubling:needed}
In the proof of our main result the following observation will be useful. If $M$ denotes the corkscrew constant for $\Omega$, it follows easily from the previous result,  Lemma \ref{eqn:1.2} and Harnack's inequality that 
\begin{equation}\label{doubling:needed}
\omega_L^X(\Delta(q,2r)) \leq C_2\omega_L^X(\Delta(q,r)),
\end{equation}
for every $q\in\partial\Omega$, $0<r<\diam(\partial\Omega)$ and for all $X\in\Omega$ with $\delta(X)\ge r/(2M)$. Here $C_2$ is a constant that depends on the allowable parameters associated with $\Omega$ and the ellipticity constants of $L$. 

\end{remark}

Our next result establishes that if a domain satisfies the Harnack chain condition then we can modify the chain of balls so that they avoid a non-tangential balls inside:

\begin{lemma}\label{lemm:NC-avoid}
Let $\Omega\subset \mathbb{R}^n$ be an open set satisfying the Harnack chain condition with constants $M, C_1>1$. Given $X_0\in\Omega$, let $B_{X_0}=B(X_0,\delta(X_0)/2)$. For every $X,Y\in \Omega\setminus \overline{B_{X_0}}$, if  we set $\Pi=|X-Y|/\min\{\delta(X), \delta(Y)\}$, then 
there is a chain of open Harnack balls $B_1, B_2, \dots, B_K\subset \Omega$ with $K \leq  100 (M+C_1^2)(2+\log_2^+ \Pi)$ that connects $X$ to $Y$. Namely, $X\in B_1$, $Y\in B_K$, $B_k\cap B_{k+1}\neq\emptyset$ for every $1\le k\le K-1$ 
and for every $1\le k\le K$
\begin{equation}\label{preHarnackball:new}
 	(100\,C_1)^{-2} \diam(B_k) \leq \dist(B_k,\partial\Omega) \leq 100\,C_1^2 \diam(B_k).
\end{equation}
Moreover, $B_k\cap \frac12 B_{X_0}=\emptyset$ for every $1\le k\le K$.
\end{lemma}

\begin{proof}
Fix $X, Y$ as in the statement and without loss of generality we assume that $\delta(X)\le\delta(Y)$. Use the Harnack chain condition for $\Omega$ to construct the chain of balls $B_1,\dots, B_K$ as in Definition \ref{def:HCC}. If none of $B_k$ meets $B_{X_0}$ then there is nothing to do as this original chain satisfies all the required condition.
Hence we may suppose that some $B_k$ meets $B_{X_0}$. The main idea is that then we can modify the chain of balls by adding some small balls that surround $X_0$. To be more precise, we let $k_-$ and $k_+$ be respectively the first and last ball in the chain meeting $B_{X_0}$. Note that $1\le k_-\le k_+\le K$.

We pick $X_{-}\in B_{k_-}\setminus\overline{B_{X_0}}$:  If $k_-=1$ we let $X_{-}=X$ or if $k_->1$ we pick $X_{-}\in B_{k_--1}\cap B_{k_-}$. Since $B_{k_-}$ meets $B_{X_0}$ then we can find  $Y_{-}\in B_{k_-}\cap\partial B_{X_0}$ such that the open  segment joining $X_{-}$ and $Y_{-}$ is contained in $B_{k_-}\setminus\overline{B_{X_0}}$. Analogously we can find $X_{+}\in B_{k_+}\setminus\overline{B_{X_0}}$ and $Y_{+}\in B_{k_+}\cap\partial B_{X_0}$ such that the open  segment joining $X_{+}$ and $Y_{+}$ is contained in $B_{k_+}\setminus\overline{B_{X_0}}$.

Next set $r=\delta(X)/(16 C_1)$ and let $N_\pm \ge 0$ be such that $N_\pm \le |X_{\pm}-Y_{\pm}|/r < N_\pm +1$. For $j=0,\dots, N_\pm$, let 
$$
B_\pm^j=B(X_{\pm}^j,r),
\qquad
\text{where}\quad
X_{\pm}^j= X_\pm+jr \frac{Y_{\pm}-X_{\pm}}{|Y_{\pm}-X_{\pm}|}
$$
Straightforward arguments show that $N_\pm\le 32 C_1^2$, $X_\pm\in B_\pm^0$, $Y_\pm\in B_\pm^{N_\pm}$, 
$B_\pm^j\cap B_\pm^{j+1} \neq\emptyset$ for every $0\le j\le N_\pm-1$,  and
$$
(32 C_1^2)^{-1}\diam(B_\pm^j)
\le
\dist(B_\pm^j, \partial\Omega)
\le
32 C_1^2\diam(B_\pm^j),
\qquad
B_\pm^j\cap\frac12 B_{X_0}=\emptyset,
$$
for every $0\le j\le N_\pm-1$.

Next, since $X_{\pm}\in\partial B_{X_0}$ we can find a sequence of balls $B^0,\dots, B^{N}$ centered at $\partial B_{X_0}$ and with radius $\delta(X)/16$ (hence $B^j\cap\frac12 B_{X_0}=\emptyset$) so that $N\le 64$, $Y_-\in B^0$, $Y_+\in B^N$, $B^j\cap B^{j+1}\neq\emptyset$ for $0\le j\le N-1$ and $32^{-1}\le \dist(B^j, \partial\Omega)/\diam(B^j)\le 32$.

Finally, to form the desired Harnack chain we concatenate the sub-chains 
$\{B_1, \dots B_{k_--1}\}$, $\{B_{-}^0,\dots B_-^{N_-}\}$, $\{B^0,\dots B^N\}$, $\{B_+^{N},\dots, B_+^0\}$, $\{B_{k_++1},\dots B_K\}$ and the resulting chain have all the desired properties. To complete the proof we just need to observe that the length of the chain is controlled by $K+N_-+N+N_++3\le 100 (M+C_1^2)(2+\log_2^+ \Pi)$.
\end{proof}

The reader may be familiar with the notion of convergence of compact sets in the Hausdorff distance; for general closed sets, not necessarily compact, we use the following notion of convergence, see \cite[Section 8.2]{DS1} for details. (It was pointed out to us that this notion is also referred to as the Attouch-Wets topology, see for example \cite[Chapter 3]{Be}.)
\begin{defn}[Convergence of closed sets]\label{def:cvsets}
	Let $\{E_j\}$ be a sequence of non-empty closed subsets of $\RR^n$, and let $E$ be another non-empty closed subset of $\RR^n$. We say that $E_j$ converges to $E$, and write $E_j \to E$, if
	\[ \lim_{j\to \infty} \sup_{x\in E_j \cap B(0,R)} \dist(x, E) = 0 \]
	and
	\[ \lim_{j\to \infty} \sup_{x\in E \cap B(0,R)} \dist(x, E_j) = 0 \]
	for all $R>0$. By convention, these suprema are interpreted to be zero when the relevant sets are empty.
\end{defn}

We remark that in the above definition, we may replace the balls $B(0,R)$ by arbitrary balls in $\RR^n$. The following compactness property has been proved in \cite[Lemma 8.2]{DS1}:
\begin{lemma}[Compactness of closed sets]\label{lm:cptHd}
	Let $\{E_j\}$ be a sequence of non-empty closed subsets of $\RR^n$, and suppose that there exists an $r>0$ such that $E_j \cap B(0,r) \neq \emptyset$ for all $j$. Then there is a subsequence of $\{E_j\}$ that converges to a nonempty closed subset $E$ of $\RR^n$ in the sense defined above.
\end{lemma}

Given a Radon measure $\mu$ on $\R^n$ (i.e., a non-negative Borel such that the measure of any compact set is finite) we define
$$
\spt\mu = \overline{\big\{ x\in\RR^n: \mu(B(x,r)) > 0 \text{ for any }r>0 \big\}}.
$$

\begin{defn}\label{def:mu-ADR}
We say that a Radon measure $\mu
$ on $\R^n$  is Ahlfors regular with constant $C\ge 1$, if there exits a constant $C\ge 1$ such that for any $x\in E$ and $0<r<\diam(E)$,
	\[ C^{-1}\, r^{n-1} \leq \mu(B(q,r)) \leq C\, r^{n-1},
	\qquad
	\forall\,x\in\spt\mu,\ 0<r<\diam(\spt\mu).	
	\]
\end{defn}

\begin{defn}\label{def:wcvm}
	Let $\{\mu_j\}$ be a sequence of Radon measures on $\RR^n$. We say $\mu_j$ converge weakly to a Radon measure $\mu_{\infty}$ and write $\mu_j \rightharpoonup \mu_{\infty}$, if 
	\[ \int f d\mu_j \to \int f d\mu_{\infty} \]
	for any $f\in C_c(\R^n)$.
\end{defn}

We finish this section by stating a compactness type lemma for Radon measures which are uniformly doubling and ``bounded below''.

\begin{lemma}[{\cite[Lemma 2.19]{TZ}}]\label{lm:sptcv}
Let $\{\mu_j\}_j$ be a sequence of Radon measures. Let $A_1, A_2>0$ be fixed constants, and assume the following conditions:
\begin{enumerate}[label=\textup{(\roman*)}, itemsep=0.2cm] 
		\item\label{1-lm:sptcv} $0\in \spt\mj$ and $\mj(B(0,1)) \geq A_1$ for all $j$,
		
		\item\label{2-lm:sptcv} For all $j\in\NN$, $q\in \spt \mj$ and $r>0$,
			\begin{equation}\label{unif-doubling}
				\mu_j(B(q,2r))\le A_2\mu_j(B(q,r))
			\end{equation}
\end{enumerate}
If there exists a Radon measure $\minf$ such that $\mj \rightharpoonup \minf$, then $\mu_\infty$ is doubling and
	\begin{equation}\label{eqn:spt-conv}
	 \spt \mj \to \spt \minf,
	 	 \end{equation}
	 	 in the sense of Definition \ref{def:cvsets}.
\end{lemma}

\section{Compactness argument}\label{comp-tt}

To prove the Main Theorem (see page \pageref{thm:main}) we will proceed by 
contradiction. First we discuss the constant $N$. 
Recall that, as noted above, the assertion that \ref{1-thm:hmu} implies \ref{3-thm:hmu} in Theorem \ref{thm:hmu}
extends routinely to all constant coefficient second order elliptic operators; alternatively, this fact follows
from the results of \cite{HMT1}
as \eqref{HMT-CM} and \ref{H1} holds trivially in the constant coefficient case.
Thus given values of the allowable constants 
$M, C_1, C_{AR},\Lambda/\lambda$, 
$C_0,\theta$, 
let $\Omega\subset\mathbb{R}^n$, $n\ge 3$, be a uniform domain with constants $M, C_1$, 
whose boundary is Ahlfors regular with constant $C_{AR}$, and let
$L = -\divg(\mathcal{A}_0\nabla)$ be a constant coefficient elliptic operator where the constant 
real symmetric matrix 
$\mathcal{A}_0$ satisfies \eqref{def:UE} with ellipticity constants $\lambda, \Lambda$,
and such that the corresponding elliptic measure $\omega_{L} \in A_{\infty}(\sigma)$ with 
constants $C_0$ and $\theta$.  Then
there exists a constant $N_0=N_0(M, C_1, C_{AR},\Lambda/\lambda,C_0,\theta)$  
such that  $\Omega$ satisfies the exterior corkscrew condition 
with constant $N_0$. We underline that this $N_0$ depends on the ratio of the ellipticity constants rather than the matrix $\mathcal{A}_0$ per se. 

With this in mind, set
 \begin{equation}\label{N}
 N=4N_0(4M,2C_1, 2^{5(n-1)}C_{AR}^2,\Lambda/\lambda, C_0  C_2 C_{AR}^{4\theta} 2^{8(n-1)\theta},\theta)
 \end{equation}
 where the constant 
$C_2=C_2(M,C_1, C_{AR}, \Lambda/\lambda)$ can be found in Remark \ref{rem:doubling:needed}.


We now state the contradiction hypothesis: 
for fixed $n\ge 3$, we
suppose that there exists a set of allowable constants $M, C_1, C_{AR}>1$, 
$\Lambda \ge \lambda = 1$,
$C_0>1$ and $0<\theta<1$, and a sequence $\epsilon_j$ (with $\epsilon_j \to 0$ as $j\to\infty$), so that the following holds:


\begin{enumerate}[label=\textup{\textbf{Assumption (\alph*):}},ref=\textup{\textbf{Assumption (\alph*)}}, itemsep=0.2cm, wide, leftmargin=1cm] 

	\item\label{1-assump} For each $j$ there is a bounded domain $\Omega_j\subset \mathbb{R}^n$, which is  uniform with constants $M, C_1$ and whose boundary is Ahlfors regular with constant $C_{AR}$. Also, there is an elliptic matrix  $\mathcal{A}_j$ defined on $\Oj$, with ellipticity constants $\lambda=1$ and $\Lambda$, and we write $L_j = -\divg(\mathcal{A}_j\nabla)$.
	

	\item\label{2-assump} $\C(\Omega_j,\mathcal{A}_j)<\epsilon_j$ (see \eqref{def:smallCarleson}).

	\item\label{3-assump} The elliptic measure of the operator $L_j$ in $\Omega_j$ is of class $A_{\infty}$ with respect to the surface measure $\sj=\HH^{n-1}|_{\partial\Omega_j}$ with constants $C_0$ and $\theta$ (see Definition \ref{def:AinftyHMU}).

\item[\textup{\textbf{Contrary to conclusion:}}] For each $j$ there is $q_j\in\partial\Omega_j$ and $0<r_j<\diam(\pOj)$ such that $\Omega_j$ has no exterior corkscrew point with constant $N$ (as in \eqref{N}). That is, there is no ball of radius $r_j/N$ contained in $B(q_j,r_j)\setminus\overline{\Omega_j}$.
\end{enumerate}


Our goal is to obtain a contradiction and as a consequence our Main Theorem will be proved.
Without loss of generality we may assume $q_j = 0$ and $r_j = 1$ for all $j$, hence $\diam(\partial\Omega_j)>1$. Otherwise,  we just replace the domain $\Omega_j$ by $(\Omega_j - q_j)/r_j$, and replace the elliptic matrix $\wcalA_j(\cdot)$ by $\wcalA_j(q_j + r_j \cdot)$. Note that the new domain and matrix have the same allowable constants, in particular the corresponding $A_{\infty}$ constants stay the same by the scale-invariant nature of Definition \ref{def:AinftyHMU}; moreover after rescaling, the above \ref{2-assump} is still satisfied:
\[ \mathcal{C}\left(\frac{\Omega_j - q_j}{r_j}, \wcalA_j(q_j + r_j \cdot) \right) = \mathcal{C}(\Omega_j, \wcalA_j)<\epsilon_j. \]

\section{Limiting domains}\label{sect:blowup}
We want to use a compactness argument similar to the blow-up argument in \cite{TZ}. The crucial difference is that in \cite{TZ}, the elliptic operator tends to a constant-coefficient operator as we zoom in on the boundary and blow up the given domain; whereas here we need to work with a \textit{sequence} of domains and their associated elliptic operators. In particular the geometric convergence of domains does not come for free, and more work is needed to analyze the limiting domain.

To be more precise, getting to the point where we can apply Theorem \ref{thm:hmu} (more precisely, its extension to the elliptic operators with constants coefficients or alternatively \cite{HMT1} applied again to constant coefficient operators) requires showing first that if $\Oinf$ is a ``limiting domain''  of the domains $\{\Omega_j\}$'s, then $\Oinf$ is an unbounded or bounded uniform domain with Ahlfors regular boundary. To accomplish this we also need to find the limit of the Green functions. Once we have this, to show that $\omega_{L_\infty} \in A_{\infty}(\sigma_{\infty})$ for the limiting domain $\Oinf$ and the limiting operator $L_{\infty}$, we need to construct the elliptic measure $\omega_{L_\infty}^Z$ for any $Z\in\Oinf$ as a limiting measure compatible with the procedure. We will also show that $L_\infty$ is an elliptic operator with constants coefficients.

Throughout the rest of paper we follow the following conventions in terms of notations:
\begin{itemize}\itemsep=0.2cm
\item For any $Z\in\Omega_j$ we write $\delta_j(Z) = \dist(Z,\partial\Omega_j)$.

\item For any $q\in\partial\Omega_j$ and $r\in (0, \diam( \partial\Omega_j))$, we use $A_j(q,r)$ to denote a corkscrew point in $\Omega_j$ relative to  $B(q,r)\cap\partial\Omega_j$, i.e.,
		\begin{equation} \label{nta-pt}
			B\left(A_j(q,r), \frac{r}{M} \right) \subset B(q,r) \cap \Oj.
		\end{equation}

\end{itemize}

\subsection{Geometric limit}
Since $\diam(\pOj) > 1$, modulo passing to a subsequence, one of the following two scenarios occurs:

\begin{enumerate}[label=\textup{\textbf{Case \Roman*:}} , ref=\textup{\textbf{Case \Roman*}}, itemsep=0.2cm, align=left, leftmargin=2cm] 
	
\item\label{CaseI} $\diam(\Oj) = \diam(\pOj) \to \infty $ as $j\to\infty$.

\item\label{CaseII} $\diam(\Oj)=\diam(\pOj) \to R_0 \in [1,\infty)$ as $j\to\infty$.

\end{enumerate}

Therefore if $\Oj$ ``converges'' to a limiting domain $\Oinf$, respectively \ref{CaseI} and \ref{CaseII} indicate that $\Oinf$ is unbounded or bounded.

Let $X_j\in \Omega_j$ be a corkscrew point relative to  $B(0,\diam(\Oj)/2)\cap\partial\Omega_j$, then
\begin{equation}\label{eq:polej}
	|X_j| \sim \wdj(\wXj) \sim \diam(\Oj),
\end{equation} 
with constants depending on the uniform constant $M$.
Let $G_j$ be the Green function associated with $\Oj$ and the operator $L_j = -\divg(\wcalA_j \nabla)$, and $\{\oj^X\}_{X\in\Oj}$ be the corresponding elliptic measure. 
In  \ref{CaseI}  we have
\begin{equation}\label{eq:temp1}
	|\wXj| \sim \wdj(\wXj) \sim \diam(\Oj) \to \infty,
\end{equation}
i.e., the poles $\wXj$ tend to infinity eventually. We let
\begin{equation}\label{def:ujcasei}
	u_j(Z) = \frac{G_j(X_j,Z)}{\oj^{X_j}(B(0,1))}.
\end{equation}
In \ref{CaseII}, we may assume that $\diam(\Oj) \sim R_0$ for all $j$ sufficiently large (one could naively rescale again so that $R_0=1$, should that be the case one may lose the property that $r_j=1$ for all $j$). Hence, there are constants $0<c_1<c_2$ such that
\begin{equation}\label{eq:wXjcase2}
	c_1 R_0 \leq \wdj(\wXj) \le |\wXj| \leq c_2 R_0 \quad \text{ for all } j \text{ sufficiently large}.
\end{equation}
Thus modulo passing to a subsequence, $\wXj$ converges to some point $X_0$ satisfying 
\begin{equation}\label{eq:polecase2}
	c_1 R_0 \leq |X_0| \leq c_2 R_0.
\end{equation}
Note that \eqref{eq:wXjcase2} and \eqref{eq:polecase2} in particular imply that for any $\rho$ sufficiently small (depending on $R_0$ and $c_1, c_2$), the ball $B(X_0,\rho)$ is contained in $\Oj$ and $\dist(B(X_0,\rho),\pOj)\ge c_1 R_0/2$.
In this case we let
\begin{equation}\label{def:ujcaseii}
	u_j(Z) = G_j(X_j,Z).
\end{equation}

Our next goal is to describe what happens with the objects in question as we let $j\to\infty$. 
This is done in Theorems \ref{thm:pseudo-blow-geo}, \ref{thm:AW11}, \ref{thm:blow-ana-pole} below.

\begin{theorem}\label{thm:pseudo-blow-geo}
\leavevmode
Under  \ref{1-assump}, and using the notation above, we have the following properties (modulo passing to a subsequence which we relabel):
\begin{enumerate}[label=\textup{(\arabic*)}, itemsep=0.2cm] 
	\item\label{1-thm:pseudo-blow-geo} \ref{CaseI}: there is a function $u_\infty\in C(\R^n)$ such that $\uj \to u_\infty$ uniformly on compact sets; moreover $\nabla \uj\rightharpoonup\nabla u_\infty$ in $L^2_{\rm loc}(\R^n)$. 
	
\item \label{2-thm:pseudo-blow-geo} \ref{CaseII}: there is a function $u_\infty \in C(\R^n\setminus \{X_0\})$ such that $\uj \to u_{\infty}$ uniformly on compact sets in $\R^n \setminus\{X_0\}$ and $\nabla \uj \rightharpoonup \nabla u_{\infty}$ in $L_{\rm loc}^2(\R^n \setminus \{X_0\})$.

\item \label{3-thm:pseudo-blow-geo} Let $\Oinf=\{Z\in\R^n: u_\infty >0\}$\footnote{In \ref{CaseII}, see Remark \ref{rem:deawffr} part \ref{2-rem:deawffr} we extend $u_\infty$ to all of $\R^n$ by setting $u_\infty(X_0)=+\infty$.}. Then $\overline{\Oj}\to\overline{\Omega_\infty}$ and $\pOj\to \pOinf$, in the sense of Definition \ref{def:cvsets}. Moreover, $\Oinf$ is an unbounded set with  unbounded  boundary in \ref{CaseI}, and it is bounded with diameter $R_0 \geq 1$ in \ref{CaseII} .
	
	\item \label{4-thm:pseudo-blow-geo} $\Oinf$ is a nontrivial uniform domain with constants $4M$ and $2C_1$. 
	
	\item \label{5-thm:pseudo-blow-geo}  There is an Ahlfors regular measure $\mu_\infty$ with constant $2^{2(n-1)}C_{AR}$ such that $\sj \rightharpoonup \mu_\infty$. Moreover, $\spt\mu_\infty=\pOinf$. In particular, this implies that 
	\begin{equation}  \label{comp-mu-H}
	2^{-3(n-1)}C_{AR}^{-1} \minf \leq \mathcal{H}^{n-1}|_{\pOinf} \leq 2^{3(n-1)}C_{AR} \minf. 
	\end{equation}
	and hence $\partial\Omega_\infty$ is Ahlfors regular with constant $2^{5(n-1)}C_{AR}^2$.
\end{enumerate}

\end{theorem}

\begin{remark}
Note that this result is purely geometric. The proof only uses   \ref{1-assump}, which states the geometric characters of domains $\Omega_j$ (i.e., they are uniform domains with Ahlfors regular boundaries) and the ellipticity of the matrix operator $\mathcal{A}_j$. The other assumptions are irrelevant for this.
\end{remark}


\begin{proof}[Proof of \ref{1-thm:pseudo-blow-geo} in Theorem \ref{thm:pseudo-blow-geo}]
Let $R>1$ and  note that for $j$ large enough (depending on $R$) we have that $\wXj \notin B(0,4R)$ since by \eqref{eq:temp1} 
\[ 
|\wXj| = |\wXj - 0| \geq \wdj(\wXj) \sim \diam(\Oj) \to \infty,
\quad\text{as }j\to\infty.
\]	
In particular, $L_j\uj=0$ in $B(0,4R)\cap\Oj$ in the weak sense. Recall that all our domains $\Omega_j$ have Ahlfors regular boundary and hence all boundary points are Wiener regular. This in turn implies that $\uj$ is a non-negative $L$-solution on $B(0,4R)\cap\Oj$ which vanishes continuously on $B(0,4R)\cap\partial\Oj$.

On the other hand, $0\in\partial\Oj$ and, using our convention \eqref{nta-pt}, $ \wAj(0,1)$ is a corkscrew point relative to $B(0,1)\cap\partial\Oj$ in the domain $\Oj$.  
Thus, by Lemma \ref{CFMS} 
			\begin{equation}\label{eq:pjzj}
				\uj(\wAj(0,1)) \sim 1.  
			\end{equation} 
We can then invoke Lemma \ref{harn-princ}, the fact that $\wAj(0,2R)\in\Oj$ is a corkscrew point relative to  $B(0,2R)\cap\partial\Omega_j$ for the domain $\Oj$,  Harnack's inequality, and \eqref{eq:pjzj} to obtain 
			\begin{equation} \label{uj-loc-bdd}
			\sup_{Z\in \Oj\cap B(0,2R)} \uj(Z) 
			\leq C \uj(\wAj(0,2R))
			\le C_R \uj(\wAj(0,1)) 
			\leq 
			C_R.  
			\end{equation}
			Extending $\uj$ by 0 outside of $\Oj$ we conclude that the sequence $\{\uj\}_{j\geq j_0}$ is uniformly bounded in $\overline{B(0,R)}$ for some $j_0$ large enough. Since for each $j$, 
$\wcalA_{j}$ has ellipticity constants bounded below by $\lambda=1$ and above by $\Lambda$, and $\Oj$ is uniform and satisfies the CDC (as $\pOj$ is Ahlfors regular) with the same constants as $\Omega_j$, then combining Lemma \ref{lem:vanishing} with the DeGiorgi-Nash-Moser estimates we conclude that the sequence $\{\uj\}_j$ is equicontinuous on $\overline{B(0,R)}$ (in fact uniformly H\"older continuous with same exponent). 
			Using Arzela-Ascoli combined with a diagonalization argument applied on a sequence of balls with radii going to infinity, we produce $u_\infty\in C(\R^n)$ and a subsequence (which we relabel) such that $\uj \to u_\infty$ uniformly on compact sets of $\RR^n$.

As observed before,  $\uj$ is a non-negative $L$-solution on $B(0,4R)\cap\Oj$ which vanishes continuously on $B(0,4R)\cap\partial\Oj$ and which has been extended by 0 outside of $\Oj$. Thus it is a positive $L$-subsolution on $B(0,4R)$ and we can use Caccioppoli's inequality along with \eqref{uj-loc-bdd} to conclude that
\begin{equation}\label{eqn:3.1A}
\int_{B(0,R)}|\nabla \uj|^2\, dZ 
\le
C\,R^{-2} \int_{B(0,2R)}|\uj|^2\, dZ
\le
C_R. 
\end{equation}
This and \eqref{uj-loc-bdd} allow us to conclude that
\begin{equation}\label{eqn:3.4}
	\sup_j \|\uj\|_{W^{1,2}(B(0,R))} \leq C_R <\infty.
\end{equation}		
Thus, there exists a subsequence (which we relabel) which converges weakly in $W^{1,2}_{\rm loc}(\RR^n)$. Since we already know that $\uj \to u_\infty$ uniformly on compact sets of $\RR^n$, we can use again \eqref{uj-loc-bdd} to easily see that $u_\infty\in W^{1,2}_{\rm loc}(\RR^n)$, 
 and $\nabla \uj \rightharpoonup \nabla u_{\infty} $ in $L^2_{\rm loc}(\RR^n)$. This completes the proof of \ref{1-thm:pseudo-blow-geo} in Theorem \ref{thm:pseudo-blow-geo}.	
\end{proof}

\begin{proof}[Proof of \ref{2-thm:pseudo-blow-geo} in Theorem \ref{thm:pseudo-blow-geo}]
Recall that in this case $\wXj\to X_0$ as $j\to\infty$. For any $0<\rho\le c_1R_0/2$ and for all $j$ large enough we have
\begin{equation}\label{poleawaycase2}
	B\left(\wXj,\frac{\rho}{2} \right) 
	\subset 
	B(X_0,\rho) 
	\subset 
	B(\wXj,2\rho)
	\subset 
	\overline{B(\wXj,2\rho)} 
	\subset 
	\overline{B(\wXj,\wdj(\wXj)/2)}
	\subset \Oj, 
\end{equation} 
where we have used \eqref{eq:wXjcase2}. Moreover, for $j$ sufficiently large,
\begin{equation}\label{eq:temp3}
	\dist(B(\wXj,2\rho),\pOj) > \frac{c_1 R_0}{2}.
\end{equation}
For any $Z\in \Oj \setminus B(\wXj,\rho/4)$, using \eqref{def:ujcaseii} and \eqref{eqn:gub} it follows that
\begin{equation}\label{eq:ujboundcase2}
	\uj(Z) \leq \frac{C}{|Z-\wXj|^{n-2}} \leq \frac{4^{n-2}C}{\rho^{n-2}}.
\end{equation}
Extending $\uj$ by $0$ outside $\Oj$ the previous estimate clearly holds for every $Z\in \R^n\setminus\Oj$. Thus $\sup_j\|\uj\|_{L^\infty(\R^n \setminus B(X_0,\rho))}\le C(\rho)$. Moreover, as in \ref{CaseI}, the sequence is also equicontinuous (in fact uniformly H\"older continuous). Using Arzela-Ascoli theorem with a diagonalization argument, we can find $u_\infty \in C(\R^n\setminus \{X_0\})$ and a subsequence (which we relabel) such that $\uj \to u_{\infty}$ uniformly on compact sets of $\R^n\setminus\{X_0\}$. 

Let $0<R\leq \sup_{j\gg 1} \diam(\Oj) \sim R_0$. We claim that
\begin{equation}\label{eq:L2gradcase2}
	\int_{B(0,R) \setminus B(X_0,\rho)} |\nabla \uj|^2 dZ \leq C(R,\rho) <\infty.
\end{equation}

To prove this, we first  take arbitrary $q\in\pOj$ and $s$ such that $0<s\le  \wdj(\wXj)/5 \sim R_0$. In particular, if $0<\rho<c_1 R_0/10\leq \wdj(\wXj)/10$ it follows that $B(q,4s)\subset\R^n\setminus B(\wXj,2 \rho)\subset 
\R^n\setminus B(X_0,\rho)$. Thus, proceeding as in \ref{CaseI}, $\uj$ is non-negative subsolution on $B(q, 2s)$ and we can use Caccioppoli's inequality and \eqref{eq:ujboundcase2} to obtain 
\begin{align}\label{eq:L2gradbd}
	\int_{B(q,s)\setminus B(X_0,\rho)} |\nabla \uj|^2 dZ 
	=
	\int_{B(q,s)} |\nabla \uj|^2 dZ 
		\leq \frac{C}{s^2} \int_{B(q,2s )} |\uj(Z)|^2 dZ \lesssim \frac{s^{n-2}}{\rho^{2(n-2)}}.
\end{align}

Note that the previous estimate, with $q=0$ and $s=R$, gives our claim \eqref{eq:L2gradcase2} when $0<R\leq \wdj(\wXj)/5$. 

Consider next the case $R_0\sim \wdj(\wXj)/5 < R\le \sup_{j\gg 1} \diam(\Oj)\sim R_0$.  Note first that the set $\Theta_j:=\{Z\in \Oj: \wdj(Z) < \wdj(\wXj)/25\}$ can be covered by a family of balls $\{B(q_i,\wdj(\wXj)/5)\}_i$ with $q_i\in\partial\Omega$ and whose cardinality is uniformly bounded (here we recall that $\wdj(\wXj)\sim  \diam(\Oj)$), Thus, \eqref{eq:L2gradbd} applied to these each ball in the family yields
\begin{equation}\label{eq:temp4}
	\int_{\left( B(0,R) \setminus B(X_0,\rho) \right) \cap \Theta_j } |\nabla \uj|^2 dZ \leq 
\sum_i
	\int_{B(q_i,\wdj(\wXj)/5) \setminus B(X_0,\rho)} |\nabla \uj|^2 dZ \leq 
	C(R, \rho) <\infty.
\end{equation} 
On the other hand, the set $\{Z\in \Oj\setminus B(\wXj,\rho/2): \wdj(Z) \geq \wdj(\wXj)/25\}$ can be covered by a family of balls $\{B_i\}_i$ so that $r_{B_i}=\rho/16$, $4B_i \subset\Oj\setminus 
B(\wXj,\rho/4)$. Moreover, the cardinality of the family is uniformly bounded depending on dimension and the ratio $\diam(\Oj)/\rho\sim R_0/\rho$. Using \eqref{poleawaycase2}, Caccioppoli's inequality in each $B_i$ since $4B_i \subset\Oj\setminus B(\wXj,\rho/4)$, and \eqref{eq:ujboundcase2} we obtain 
\begin{equation}\label{eq:temp5}
	\int_{\left( B(0,R) \setminus B(X_0,\rho) \right) \setminus \Theta_j } |\nabla \uj|^2 dZ 
	\leq 
	\sum_i\int_{B_i} |\nabla \uj|^2 dZ 
	\lesssim 
	\sum_i\frac{1}{r_{B_i}^2} \int_{2B_i} | \uj(Z)|^2 dZ 
	\leq C(R,\rho).
\end{equation}
Combining \eqref{eq:temp4} and \eqref{eq:temp5} we obtain the desired estimate and hence proof of the claim \eqref{eq:L2gradcase2} is complete.

Next, we combine \eqref{eq:L2gradcase2} with the fact that $\sup_j\|\uj\|_{L^\infty(\R^n \setminus B(X_0,\rho))}\le C(\rho)$ to obtain 	that $\sup_j \|\uj\|_{W^{1,2}(B(0,R)\setminus B(X_0,\rho))} \leq C(R,\rho) <\infty$. Thus, there exists a subsequence (which we relabel) which converges weakly in $W^{1,2}_{\rm loc}(\RR^n\setminus B(X_0,\rho))$. Since we already know that $\uj \to u_\infty$ uniformly on compact sets of $\RR^n\setminus B(X_0,\rho)$, we can easily see that $u_\infty\in W^{1,2}_{\rm loc}(\RR^n\setminus B(X_0,\rho))$,  and $\nabla \uj \rightharpoonup \nabla u_{\infty} $ in $L^2_{\rm loc}(\RR^n\setminus B(X_0,\rho))$. This completes the proof of \ref{2-thm:pseudo-blow-geo} in Theorem \ref{thm:pseudo-blow-geo}.	
\end{proof}

\medskip

\begin{remark}\label{rem:deawffr}

In the \ref{CaseII} scenario the following remarks will become useful later. In what follows we assume that  
$0<\rho\le c_1R_0/2$ and $j$ is large enough.
\begin{enumerate}[label=\textup{(\roman*)}, itemsep=0.2cm]  
	
	\item\label{1-rem:deawffr} Let us pick $Y\in \partial B(\wXj,3\delta_j(X_j)/4)$ and note that \eqref{eq:wXjcase2} gives $Y, \wAj(0,c_1 R_0/2) \in \Oj \setminus \overline{B(\wXj,\delta_j(X_j)/2)}$, $|Y-\wAj(0,c_1 R_0/2)|<(c_1+2c_2)R_0$, and $\delta_j(Y)\ge c_1 R_0/4$. Recalling that $\Omega_j$ satisfies the interior corkscrew condition with constant $M$, it follows by definition that $\delta_j(\wAj(0,c_1 R_0/2))\ge c_1 R_0/(2M)$.  All these allow us to invoke Lemma \ref{lemm:NC-avoid} to then use \eqref{eqn:glb} and \eqref{eq:wXjcase2} and eventually show 
		\begin{equation}\label{eq:ujfixedvalue}
			\uj \left( \wAj\left(0,\frac{c_1 R_0}{2} \right) \right)
			\sim \uj(Y)
			\gtrsim \left| Y - \wXj \right|^{2-n} \sim \delta_j(X_j)^{2-n}\sim R_0^{2-n},
		\end{equation}
		where the implicit constants are independent of $j$.
		
		\item\label{2-rem:deawffr} The set $\partial B(X_0,\rho)$ is compact and away from $X_0$, so $\uj \to u_\infty$ uniformly in $\partial B(X_0,\rho)$. Since $\wXj \to X_0$, for any $Z\in \partial B(X_0,\rho)$ we have
			$\rho/2< |Z-\wXj| < 2\rho$ for $j$ sufficiently large. In particular by choosing $\rho<R_0/(16M)$, we have for $j$ large enough
			\begin{equation}
				|Z-\wXj| < 2\rho < \frac{R_0}{8M} \leq \frac{\diam(\Oj)}{4M} \leq \frac{\wdj(\wXj)}{2}, 
			\end{equation}
			where the last estimate uses that $X_j\in \Omega_j$ is a corkscrew point relative to the surface ball  $B(0,\diam(\Oj)/2)\cap\partial\Omega_j$ with constant $M$. Thus by \eqref{eqn:glb} if $j$ is large enough
\[ 
		\uj(Z) \gtrsim |Z-\wXj|^{2-n} \gtrsim \rho^{2-n}, \qquad  \forall\,Z\in \partial B(X_0,\rho) 
\]
			with implicit constants which are independent of $j$. Therefore,
			\begin{equation}\label{eq:uinftynearpole}
				u_\infty(Z) = \lim_{j\to\infty} \uj(Z) \gtrsim \rho^{2-n},
								\qquad  \forall\,Z\in \partial B(X_0,\rho) 
			\end{equation}
			For this reason it is natural to extend the definition of $u_\infty$ to all of $\R^n$ by simply letting $u_\infty(X_0) =+ \infty$.
			
			\item\label{3-rem:deawffr} Since $\uj$ is the Green function in $\Oj$ for $L_j$, an 
			elliptic operator with uniform ellipticity constants $\lambda =1$ and $\Lambda$, 
			by \eqref{eqn:glq} we know for any $1<r<\frac{n}{n-1}$,
				\begin{equation}\label{eq:upperr}
					\|\nabla \uj\|_{L^r(\Oj)} \lesssim |\Oj|^{\frac1r-\frac{n-1}n}\lesssim 
					R_0^{\frac{n}r-n+1}<\infty,
				\end{equation}
				provided $j$ is large enough and where the implicit constants depend on 
				dimension, $r$, and $\Lambda$, but are independent of $j$. 
				Note that $\nabla \uj \equiv 0$ outside of $\Oj$ by construction. Thus, one can easily show that passing to a subsequence (and relabeling) $\nabla \uj \rightharpoonup \nabla u_{\infty}$ in $L^r_{\rm loc}(\RR^n)$ for $1<r<n/(n-1)$.
			
	\end{enumerate}
\end{remark}

\medskip

\begin{proof}[Proof of \ref{3-thm:pseudo-blow-geo} in Theorem \ref{thm:pseudo-blow-geo}: \ref{CaseI}]

	It is clear that $\Oinf$ is an open set in \ref{CaseI} since $u\in C^\infty(\R^n)$.
	On the other hand, 	since $0\in\pOj$ for all $j$, by Lemma \ref{lm:cptHd} and modulo passing to a subsequence (which we relabel) we have that there exist non-empty closed sets $\Gamma_{\infty}, \Lambda_{\infty}$ such that
			$\overline{\Oj}\to \Gamma_\infty$ and $\pOj\to \Lambda_\infty$ as $j\to\infty$, where the convergence is in the sense of Definition \ref{def:cvsets}.
			
We are left with obtaining 
\begin{equation}\label{eq:claim-boundary--}
\Lambda_\infty = \pOinf
\qquad\mbox{and}\qquad\Gamma_\infty=\overline\Omega_\infty.
\end{equation}		
We first show that $\Lambda_\infty\subset \partial\Omega_\infty$. To that end we take $p\in\Lambda_\infty$, and there is a sequence $p_j\in\pOj$ such that $\lim_{j\to\infty} p_j = p$. Note that
			$ u_\infty(p) =\lim_{j\to \infty}\ \uj(p)$. On the other hand since the $\uj$'s are uniformly H\"older continuous on compact sets (see the Proof of  \ref{1-thm:pseudo-blow-geo} in Theorem \ref{thm:pseudo-blow-geo}) and $\uj(p_j)=0$ as $p_j\in\partial\Oj$ we have
			$$
			0\le u_\infty(p)
			\le
			|u_\infty(p)-u_j(p)|+
			|u_j(p)-u_j(p_j)|
			\lesssim
			|u_\infty(p)-u_j(p)|+|p-p_j|^\alpha\to 0,
			$$
as $j\to\infty$. Thus $u_{\infty}(p) = 0$, that is, $p\in \R^n\setminus \Omega_\infty$. 

Our goal is to show that $p\in\partial\Omega_\infty$. Suppose that $p\notin\partial\Omega_\infty$, then $p\in\R^n\setminus\overline{\Omega_\infty}$ and there exists $\epsilon\in(0,1)$ such that $B(p,\epsilon) \subset \R^n\setminus\overline{\Omega_\infty}$, that is, $u_\infty\equiv 0$ on $\overline{B(p,\epsilon)}$. In $\Oj$ we have
			$$
				\left|A_j\left(p_j,\frac{\epsilon}{2}\right) - A_j(0,1)\right|  \le \frac{\epsilon}{2} + | p_j | + 1 
				 \le 2\left(|p| + 1\right)
			$$
			and
			$$
			 \delta_j\left( A_j\left(p_j,\frac{\epsilon}{2}\right) \right) \geq \frac{1}{M} \frac{\epsilon }{2}, 
			\qquad \delta_j\left( A_j(0,1)\right) \geq \frac{1}{M}. 
			 $$
			Note also that
			$$
			\frac{\delta_j\left( A_j\left( p_j,\frac{\epsilon }{2} \right)\right) }{\delta_j(X_j)}
			+
			\frac{\delta_j\left( A_j(0,1)\right) }{\delta_j(X_j)}
			\sim
			\frac{1}{\diam(\Omega_j)}\to 0,
			\qquad\text{as }j \to\infty,
			$$
			hence for $j$ large enough, $A_j(0,1), A_j\left(p_j,\frac{\epsilon }{2} \right) \notin \overline{B(X_j,\delta_j(X_j)/2})$. 
						
						We can then apply Lemma \ref{lemm:NC-avoid} and Harnack's inequality along the constructed chain in $\Omega_j$ to obtain 
			$$
			G_j\left( X_j,A_j\left( p_j,\frac{\epsilon }{2} \right) \right) \sim G_j(X_j,A_j(0,1)),
			$$
			where the implicit constants depend on the allowable parameters, $\epsilon$ and $|p|$, but are uniform on $j$. Hence by \eqref{eq:pjzj},
\begin{equation}\label{eq:tmplb}
				\uj\left(\wAj\left(p_j,\frac{\epsilon}{2}\right)\right)  = \dfrac{ G_j \left( X_j,A_j\left( p_j,\frac{\epsilon }{2} \right)\right)}{\omega_j^{X_j}(B(0,1))} 
				\gtrsim C \dfrac{ G_j\left( X_j,A_j(0,1) \right)}{\omega_j^{X_j}(B(0,1))}  
				 = \uj(\wAj(0,1))
\ge C_0,
				\end{equation}
where $C_0$ is independent of $j$.  

Note that since $\uj\to u_\infty$ on compact sets it follows from our assumption that for $j$ large enough  depending on $C_0$ 
\begin{equation}\label{eq:contrad}
u_j(Z)=u_j(Z)-u_\infty(Z)<\frac{C_0}2,
\qquad
\forall\,Z\in \overline{B(p,\epsilon)}.
\end{equation}
However, for $j$ large enough $\wAj (p_j,\epsilon/2) \in B(p_j,\epsilon/2)  \subset B(p,\epsilon)$  and then \eqref{eq:contrad} contradicts \eqref{eq:tmplb}. Thus, we have shown that necessarily $p\in\partial\Omega_\infty$ and consequently $\Lambda_\infty\subset \partial\Omega_\infty$.

Let us next show that $\pOinf \subset \Lambda_{\infty}$. Assume that $p \notin \Lambda_{\infty}$. Since $\Lambda_{\infty}$ is a closed set, there exists $\epsilon>0$ such that $B(p,2\epsilon)\cap\Lambda_{\infty} = \emptyset$. Since $\Lambda_{\infty}$ is the limit of $\pOj$, by Definition \ref{def:cvsets} we have that for $j$ large enough $B(p,\epsilon) \cap \pOj=\emptyset$. Hence, by passing to a subsequence (and relabeling) either $B(p,\epsilon) \subset \Oj$ for all $j$ large enough or $B(p,\epsilon) \subset \R^n\setminus\overline{\Oj}$ for all $j$ large enough. 

We first consider the case $B(p,\epsilon) \subset \Oj$. Hence,  $\delta_j(p)  \geq \epsilon $ and $|A_j(0,1) - p| \leq 1+|p|$. Thus there exists a Harnack chain joining $A_j(0,1)$ and $p$ whose length is independent of $j$ and depends on $\epsilon$ and $|p|$. 
We next observe that for $j$ large enough $|p-\wXj|>\wdj(\wXj)/2$. Indeed, if we take $j$ large enough, using that $0\in\partial\Oj$ and \eqref{eq:temp1} we clearly have
$$
1\le \frac{|\wXj|}{\wdj(\wXj)}
\le
\frac{|\wXj-p|}{\wdj(\wXj)}
+
\frac{|p|}{\wdj(\wXj)}
<
\frac{|\wXj-p|}{\wdj(\wXj)}+\frac12,
$$
and we just need to hide to obtain the desired estimate. 
Once we know that $|p-\wXj|>\wdj(\wXj)/2$, we also note that $|\delta_j(A_j(0,1))|\le 1\ll \diam(\Omega_j)\sim\delta_j(X_j)$
and hence $A_j(0,1)\notin\overline{B(X_j,\delta(X_j)/2})$ for $j$ large enough. 

We can now invoke Lemma \ref{lemm:NC-avoid} and Harnack's inequality along the constructed chain in $\Omega_j$ to obtain that $G_j(X_j,p) \sim G_j(X_j,A_j(0,1))$, which combined with \eqref{def:ujcasei} and \eqref{eq:pjzj}, yields
\begin{equation}\label{eqn:3.5}
\uj(p) \sim \uj(\wAj(0,1)) \sim 1,
\end{equation}
where the implicit constants depend on the allowable parameters, $p$ and $\epsilon$, but are uniform on $j$. Letting $j\to\infty$ we obtain that $u_\infty(p)\sim 1$ which implies that $p\in\Omega_\infty$, and since we have already shown that  $\Omega_\infty$ is open, it follows that $p\notin\partial \Omega_\infty$.
			
We next consider now the case 	$B(p,\epsilon) \subset \R^n\setminus\overline{\Oj}$ for all $j$ large enough which implies that by construction $\uj(X) = 0$ for all $X\in B(p,\epsilon)$. By uniform convergence of $\uj$ in compact sets we have that $u_{\infty}(X) = 0$ for $X\in B(p,\epsilon/2)$, which implies $B(p,\epsilon/2) \subset \{u_{\infty}=0\}$ and therefore $p\notin \partial\Oinf$.

In both cases we have shown that if $p \notin \Lambda_{\infty}$ then  $p\notin \partial\Oinf$, or, equivalently, $\partial\Oinf\subset \Lambda_{\infty}$.  This together with the converse inclusion completes the proof of $\Lambda_{\infty} = \partial\Oinf$.

Our next goal is to show that $\Gamma_\infty=\overline\Omega_\infty$. Note that if $Z\in\Omega_\infty$, then $u_\infty(Z)>0$ and this implies that $\uj(Z)>0$ for $j$ large enough. The latter forces $Z\in\Omega_j$ 	for all $j$ large enough. This implies that $Z\in \Gamma_\infty$, and we have shown that $\Omega_\infty\subset \Gamma_\infty$. Moreover since $\Gamma_{\infty}$ is closed, we conclude that $\overline\Omega_\infty\subset \Gamma_\infty$. 

To obtain the converse inclusion we take $X\in \Gamma_\infty$. Assume that there is $\epsilon>0$ such that $\overline{B(X,2\epsilon)}\subset \R^n\setminus\overline{\Omega_\infty}$, in particular $B(X,2\epsilon) \cap \partial\Oinf = \emptyset$. Since we have already shown that $\partial\Oinf$ is the limit of $\pOj$'s, for $j$ large enough $B(X,\epsilon) \cap \pOj = \emptyset$. By the definition of $\Gamma_{\infty}$, there is a sequence $\{Y_j\}\subset \overline\Oj $ with $Y_j\to X$ as $j\to\infty$. Thus, for all $j$ large enough $B(X,\epsilon)$ is a neighborhood of $Y_j$; and in particular $\Oj\cap B(X,\epsilon)\not =\emptyset$ since $Y_j \in \overline{\Oj}$. On the other hand, since $B(X,\epsilon) \cap \pOj = \emptyset$
			we conclude that $B(X,\epsilon)\subset \Oj$. At this point we follow a similar argument to the one used to obtain \eqref{eqn:3.5} replacing $p$ by $X$ and obtain for all $j$ large enough
			$$
				\uj(X) \sim \uj(\wAj(0,1)) \sim 1,
			$$
			where the implicit constants depend on the allowable parameters, $|X|$ and $\epsilon$, but are uniform on $j$. Letting  $j\to\infty$ it follows that $u_{\infty}(X)>0$ and hence $X\in\Oinf$, contradicting the assumption that there is $\epsilon>0$ such that $\overline{B(X,2\epsilon)}\subset \R^n\setminus\overline{\Omega_\infty}$. In sort, we have shown that $\overline{B(X,2\epsilon)}\cap  \overline{\Omega_\infty}\neq\emptyset$ for every $\epsilon>0$, that is, $X\in  \overline{\Omega_\infty}$. We have eventually proved that  $\Gamma_\infty\subset \overline\Omega_\infty$ this completes the proof of  \eqref{eq:claim-boundary--} in the \ref{CaseI} scenario.

			Since $\diam(\Omega_j) \to \infty$ and $0\in \overline{\Oj} \to \overline{\Oinf}$ uniformly on compact set, $\Oinf$ is unbounded. Otherwise we would have $\overline{\Oinf} \subset B(0,R)$, and for sufficiently large $j$ one would see that $\overline{\Oj} \subset B(0,2R)$, which is a contradiction.

On the other hand, it is possible that $\diam(\pOj) \not\to \diam(\pOinf)$, hence we do not know whether $\diam(\pOinf)=\infty$. However, under the assumption that the $\pOj$'s are Ahlfors regular with uniform constant, we claim that $\pOinf$ is also unbounded. Assume not, then there is $R>0$ such that $\pOinf \subset B(0,R)$. Let $k$ be a large integer, and notice that $\pOj \to \pOinf$ uniformly on the compact set $\overline{B(0,kR)}$. Thus for $j$ sufficiently large (depending on $k$) 
			\begin{equation}\label{eq:nobdbd}
				\pOj \cap \overline{B(0,kR)} \subset B(0,2R).
			\end{equation}  
			Since $\diam(\pOj)\to \infty$ we can also guarantee that $\diam(\pOj) > kR $ for $j$ sufficiently large. Recalling that $0\in \pOj$, we can then consider the surface ball $\Delta_j(0,kR) = B(0,kR)\cap \pOj$. By \eqref{eq:nobdbd} and the Ahlfors regularity of $\pOj$,
			\begin{equation}
				C_{AR}^{-1}(kR)^{n-1}\leq \sj(\Delta_j(0,kR)) \le  \sj(B(0,2R)\cap\partial\Omega_j) \leq C_{AR} (2R)^{n-1}.
			\end{equation}
			Letting $k$ large readily leads to a contradiction. 
			\end{proof}
			
\begin{proof}[Proof of \ref{3-thm:pseudo-blow-geo} in Theorem \ref{thm:pseudo-blow-geo}: \ref{CaseII}]
Take $X\in \Omega_\infty$, that is, $u_\infty(X)>0$. If $X\neq X_0$ then $u_\infty$ is continuous at $X$ and hence $u_\infty(Z)>0$ for every $Z\in B(X,r_x)$ for some $r_x$ small enough. On the other hand, if $X=X_0$, by Remark \ref{rem:deawffr} part \ref{2-rem:deawffr} we have that $u_\infty(Z)>0$ for all $Z\in B(X_0, \rho)$ with $\rho$ sufficiently small (here we use the convention that $+\infty>0$). Note that this argument show in particular that $B(X_0,\rho)\subset \Omega_\infty$.

On the other hand, 	since $0\in\pOj$ for all $j$, by Lemma \ref{lm:cptHd} and modulo passing to a subsequence (which we relabel), there exist closed sets $\Gamma_{\infty}, \Lambda_{\infty}$ such that
			$\overline\Oj\to \Gamma_\infty$ and $\pOj\to \Lambda_\infty$ as $j\to\infty$ in the sense of Definition \ref{def:cvsets}. We are going to obtain that 
\begin{equation}\label{eq:claim-boundary--CaseII}
\Lambda_\infty = \pOinf
\qquad\mbox{and}\qquad\Gamma_\infty=\overline\Omega_\infty.
\end{equation}

Let $p\in\Lambda_\infty$, there is a sequence $\{p_j\}\subset \pOj$ such that $p_j\to p$ as $j\to\infty$. Note that by \eqref{eq:wXjcase2}
$$
c_1 R_0
\le 
\wdj(\wXj) 
\le 
|\wXj-p_j|
\le
|\wXj-p|+|p-p_j|.
$$
Thus, for $j$ large enough, $|\wXj-p|> \wdj(\wXj) /2>c_1R_0/2$. In particular, $X_0\neq p$
and $ \uj(p)\to u_\infty(p)$ as $j\to\infty$. 
On the other hand since the $\uj$'s are uniformly H\"older continuous on compact sets as observed above, $|\uj(p)|=|\uj(p)-\uj(p_j)| \leq C|p-p_j|^\alpha$, thus $\uj(p)\to 0$ as $j\to\infty$. Therefore $u_{\infty}(p) = 0$, that is, $p\in \R^n\setminus\Omega_\infty$. 

Suppose now that $p\notin\partial\Omega_\infty$. Then, there exists $0<\epsilon<\wdj(\wXj) /4$ such that $\overline{B(p,\epsilon)} \subset \R^n\setminus\Oinf$, or, equivalently, $u_\infty\equiv 0$ on $B(p,\epsilon)$. Note that 
			$$
				\left|\wAj\left( p_j,\frac{\epsilon}{2}\right) - \wAj\left(0,\frac{c_1 R_0}{2} \right)\right|  \leq \frac{\epsilon}{2} + |p_j | + \frac{c_1 R_0}{2} \leq C(\epsilon, |p|, R_0).
			$$
			Also, 
			$$
			\frac{\epsilon }{2M} 
			\le
			\wdj\left( \wAj\left( p_j,\frac{\epsilon}{2}\right) \right) 
			<
			\frac{\epsilon}{2}
			<
			\frac{\wdj(\wXj)}{2}
			 $$
and, by  \eqref{eq:wXjcase2},
\begin{equation}\label{eq:fwefgaserg}
\frac{c_1 R_0}{2M}\le \wdj\left( \wAj\left(0,\frac{c_1 R_0}{2} \right)\right)<\frac{c_1 R_0}{2}
\le
\frac{\wdj(\wXj)}{2}.
\end{equation}

Notice that in particular $\wAj\left( p_j,\frac{\epsilon}{2}\right)$, $\wAj\left(0,\frac{c_1 R_0}{2} \right) \notin\overline{B(\wXj,\wdj(\wXj)/2)}$. We can now invoke Lemma \ref{lemm:NC-avoid},  Harnack's inequality along the constructed chain in $\Oj$, and \eqref{eq:ujfixedvalue} to see that
			\begin{equation} \label{eqn103case2}
			\uj \left( \wAj\left( p_j,\frac{\epsilon }{2} \right) \right) 
			\sim  
			\uj\left(\wAj\left(0,\frac{c_1 R_0}{2}\right) \right) 
			\gtrsim 
			1, 
			\end{equation}
			with implicit constant depending on the allowable parameters, $\epsilon, |p|, R_0$ but independent of $j$. On the other hand, for all $j$ large enough
			\begin{equation} \label{eqn104case2}
			\wAj \left(p_j,\frac{\epsilon}{2}\right) 
			\in B\left(p_j,\frac{\epsilon}{2}\right) 
			\subset 
			\overline{B(p,\epsilon)}
			\subset
			\R^n\setminus B(\wXj,\wdj(\wXj)/4),
						 \end{equation}
hence $u_j\to u_\infty$ uniformly on $\overline{B(p,\epsilon)}$ with $u_\infty\equiv 0$ on $\overline{B(p,\epsilon)}$. This and \eqref{eqn104case2} contradict \eqref{eqn103case2} and therefore we conclude that $p\in\pOinf$, and we have eventually obtained that $\Lambda_\infty \subset \pOinf$.
			
			To show that $\pOinf \subset \Lambda_{\infty}$, we assume that $p \notin \Lambda_{\infty}$. If $p=X_0$, then since we observed above that $B(X_0,\rho) \subset \Oinf$ (see \eqref{eq:uinftynearpole}) then $X_0 \notin \pOinf$. 
			
			Assume next that $p\neq X_0$. Since $\Lambda_{\infty}$ is a closed set and since $\wXj\to X_0$ as $j\to\infty$, there exists $\epsilon>0$ such that $B(p,2\epsilon)\cap\Lambda_{\infty} = \emptyset$ and $X_0, \wXj \notin B(p, 2 \epsilon)$ for all $j$ large enough. Moreover, since $\Lambda_{\infty}$ is the limit of $\pOj$, by Definition \ref{def:cvsets} we have that for all $j$ large enough $B(p,\epsilon) \cap \pOj=\emptyset$.  Hence, passing to a subsequence (and relabeling) either $B(p,\epsilon) \subset \Oj$ for $j$ large enough or $B(p,\epsilon) \subset \R^n\setminus\overline{\Oj}$ for $j$ large enough. 
			
			Assume first that $B(p,\epsilon) \subset \Oj$ for all $j$ large enough. We consider two subcases. Assume first that $p\notin \overline{B(\wXj,\wdj(\wXj)/2}$. Then, proceeding as before, by \eqref{eq:fwefgaserg} we can apply Lemma \ref{lemm:NC-avoid} and Harnack's inequality along the constructed chain in $\Oj$ to get 
			\begin{equation}\label{eqn:3.5case2}
				\uj(p) \sim \uj\left(\wAj\left(0,\frac{c_1 R_0}{2} \right) \right) \gtrsim 1,
				\end{equation} 
				with implicit constant depending on the allowable parameters, $\epsilon, |p|, R_0$ but independent of $j$. Suppose next that $p\in \overline{B(\wXj,\wdj(\wXj)/2)}$. In that case we can use \eqref{def:ujcaseii}, \eqref{eqn:glb}, and   \eqref{eq:wXjcase2}  to see that for all $j$  large enough
			\begin{equation}\label{eqn:3.5case2-a}
				\uj(p) \gtrsim |p-\wXj|^{2-n} \gtrsim \wdj(\wXj)^{2-n}\gtrsim (c_2 R_0)^{2-n},			
												\end{equation} 
				with implicit constants which are uniform on $j$. Combining the two cases together we have shown that $\uj(p) \gtrsim 1$ uniformly on $j$. Letting $j\to\infty$ we conclude that $
			u_\infty(p)\gtrsim 1$ and hence $p\in \Omega_\infty$, and since we have already shown that $\Omega_\infty$ is an open set we conclude that $p\notin \partial\Omega_\infty$

We now tackle the second case on which $B(p,\epsilon) \subset \R^n\setminus\overline{\Oj}$ for all $j$ large enough. In this scenario $\uj(X) = 0$ for all $X\in B(p,\epsilon)$. Since $X_0\notin B(p,2\epsilon)$, by uniform convergence of $\uj$ in $\overline{B(p,\epsilon/2)}$ we have that $u_{\infty}(X) = 0$ for $X\in B(p,\epsilon/2)$, which implies $B(p,\epsilon/2) \subset \R^n\setminus\Omega_\infty$ and eventually $p\notin \partial\Oinf$.

In both cases we have shown that if $p \notin \Lambda_{\infty}$ then  $p\notin \partial\Oinf$, or, equivalently, $\partial\Oinf\subset \Lambda_{\infty}$.  This together with the converse inclusion completes the proof of $\Lambda_{\infty} = \partial\Oinf$.

Our next task is to show that $\Gamma_\infty=\overline\Omega_\infty$. 
Let $Z\in\Oinf$ and assume first that $Z=X_0$. By \eqref{poleawaycase2} and since $\wXj\to X_0$ as $j\to\infty$ we have that $X_0 \in B(\wXj,2\rho) \subset \Oj$ for all $j$ large enough, thus $Z=X_0\in \Gamma_\infty$. On the other hand, if $Z\neq X_0$ since $u_\infty(Z)>0$ we have that $\uj(Z)>0$ for all $j$ large enough. This forces as well that $Z\in\Oj$ for $j$ all large enough and again $Z\in \Gamma_\infty$. With this  we have shown that $\Omega_\infty\subset \Gamma_\infty$. Moreover, since $\Gamma_{\infty}$ is closed we conclude as well that $\overline\Omega_\infty\subset \Gamma_\infty$. 

Next we look at the converse inclusion and take $X\in \Gamma_\infty$. Assume that $X\in \R^n\setminus\overline{\Omega_\infty}$. Thus, there is $\epsilon>0$ such that $\overline{B(X,2\epsilon)}\subset \R^n\setminus\overline{\Omega_\infty}$. In particular $B(X,2\epsilon) \cap \partial\Oinf = \emptyset$ and $B(X_0,\rho) \cap B(X,2\epsilon) = \emptyset$ (recall that we showed that $B(X_0,\rho)\subset \Oinf$).  Since we have already shown that $\partial\Oinf$ is the limit of $\pOj$'s, for $j$ large enough $B(X,\epsilon) \cap \pOj = \emptyset$. By the definition of $\Gamma_{\infty}$, there is a sequence $\{Y_j\}\subset \overline\Oj$ so that $Y_j\to X$ as $j\to\infty$. Thus, for all $j$ large enough $B(X,\epsilon)$ is a neighborhood of $Y_j$, and, in particular, $\Oj\cap B(X,\epsilon)\not =\emptyset$ since $Y_j \in \overline{\Oj}$. Besides, since $B(X,\epsilon) \cap \pOj = \emptyset$
			we conclude that $B(X,\epsilon)\subset \Oj$. Using a similar argument to the one used to obtain \eqref{eqn:3.5case2} and \eqref{eqn:3.5case2-a} we have (replacing $p$ by $X$) that
			$$
				\uj(X) \gtrsim 1
			$$
			independently of $j$ and with constants that depend on the allowable parameters, $\epsilon, |X|, R_0$. Since $u_j(X)\to u_\infty(X)$ we conclude that $u_{\infty}(X)>0$ and thus $X\in\Oinf$, contradicting the assumption that $X\in \R^n\setminus\overline{\Omega_\infty}$. Eventually, $X\in\overline\Omega_\infty$ and we have obtained that $\Gamma_\infty\subset \overline\Omega_\infty$.

			Since $\diam(\Oj) \to R_0$ is finite and $0\in \pOj$, we have $\Oj, \Oinf \subset \overline{B(0,2R_0)}$ for $j$ sufficiently large. Hence $\overline{\Oj} \to \overline{\Oinf}$ uniformly, and thus $\diam(\Oinf) = \lim\limits_{j\to\infty} \diam(\Oj) = R_0 \geq 1$. 
\end{proof}

For later use let us remark that in the \ref{CaseII} scenario the fact that 
			$\overline{\Oj}\to\overline{\Omega_\infty}$ and $\pOj \to \pOinf$ as $j\to\infty$ in the sense of Definition \ref{def:cvsets} yields
			\begin{equation}\label{diamOinf}
				\diam(\Oinf) = \diam (\overline{\Oinf}) = \lim_{j\to\infty} \diam(\overline{\Oj}) 
				= 
				\lim_{j\to\infty} \diam(\Oj)= 
				R_0.
			\end{equation}
			\begin{equation}\label{diampOinf}
				\diam(\pOinf) = \lim_{j\to\infty}\diam(\pOj) = R_0
			\end{equation}

\begin{proof}[Proof of \ref{4-thm:pseudo-blow-geo} in Theorem \ref{thm:pseudo-blow-geo}] 
Notice that $\Omega_\infty\not =\emptyset$ since $0\in\partial\Omega_\infty$. Next we show that $\Omega$ satisfies the interior corkscrew and the Harnack chain. Let us sketch the argument. For the interior corkscrew condition, fixed $p\in\pOinf$ and $0<r< \diam(\pOinf)$, we take a sequence $p_j\in \partial\Omega_j$ so that $p_j\to p$ and for each $j$ we let $A_j$ be an interior corkscrew relative to $B(p_j,r/2)\cap\partial\Omega_j$ in $\Omega_j$. All the $A_j$'s are contained in $B(p, 3r/4)$, hence, passing to a subsequence, they converge to some point $A$. 
Using that the interior corkscrew condition holds for all $\Omega_j$ with the same constant $M$, it follows that each $A_j$ is uniformly away from $\partial\Omega_j$ and so will be $A$ from $\partial \Omega_\infty$ since $\pOj \to \pOinf$. I turn, this means that $A$ is an interior corkscrew relative to $B(p,r)\cap\partial\Omega_\infty$ in $\Omega_\infty$. Regarding the Harnack chain condition we proceed in a similar fashion. Fixed $X, Y\in\Omega_\infty$ for some fixed $j$ large enough we will have that $X,Y\in\Omega_j$ with $\delta_j(X)\approx \delta_\infty(X)$ and $\delta_j(Y)\approx \delta_\infty(Y)$. We can then construct a Harnack chain to join $X$ and $Y$ within  $\Omega_j$ (whose implicit constants are independent of $j$). Again, since each ball in the constructed Harnack chain is uniformly away from $\partial \Omega_j$, it will also be uniformly away from $\partial\Omega_\infty$  allowing us to conclude that this chain of balls is indeed a Harnack chain within $\Omega_\infty$.

\noindent\textbf{Interior corkscrew condition.} Recall that each $\Omega_j$ is a uniform domain with constants $M, C_1>1$. Hence, for all $q\in \pOj$ and $r\in (0,\diam(\pOj))$ there is a point $A_j(q,r)\in \Omega_j$ such that
		\begin{equation}\label{eqn106}
B\left(A_j(q,r), \frac{r}{M} \right) \subset B(q,r)\cap\Omega_j.
	\end{equation} 
			 
Let $p\in\pOinf$ and $0<r< \diam(\pOinf)$. In \ref{CaseII}, by \eqref{diampOinf} we get that $r<\diam(\pOj)$ for all $j$ sufficiently large. In \ref{CaseI}, either $\diam(\pOinf) = \infty$ or $\diam(\pOinf) < \infty$, but we still have $r<\diam(\pOj)$ for all $j$ sufficiently large (note that in the latter case $\diam(\pOj) \not\to \diam(\pOinf)$).	
	Since $\pOj\to\pOinf$, we can find $p_j\in \pOj$ converging to $ p$. For each $j$ there exists $\wAj(p_j,r/2)$ such that 
	\begin{equation}\label{eqn107}
		B\left(\wAj\left(p_j, \frac{r}{2}\right), \frac{r}{2M}\right) \subset B\left(p_j, \frac{r}{2}\right) \cap \Oj.
	\end{equation}
	In particular we deduce that
	\begin{equation}\label{saegt}
		\overline{B\left(\wAj\left(p_j, \frac{r}{2}\right), \frac{r}{3M}\right)} \subset \Oj
		\qquad\mbox{and}\qquad 
		\dist\left( B\left(\wAj\left(p_j, \frac{r}{2}\right), \frac{r}{2M}\right), \pOj \right) \geq \frac{r}{6M}.
	\end{equation}
Note that for $j$ large enough 
\begin{equation} \label{eqn108}
\wAj\left(p_j, \frac{r}{2}\right) \in B\left( p_j, \frac{r}{2} \right) \subset \overline{B\left(p,\frac{3r}{4}\right)}. 
\end{equation}
Modulo passing to a subsequence (which we relabel) $\wAj\left(p_j, r/2\right)$ converges to some point, which we denote by $A(p,r)$, and  for all $j$ sufficiently large (depending on $r$)
\begin{equation} \label{eqn109}
B\left(A(p,r), \frac{r}{4M} \right) \subset B\left(\wAj\left(p_j, \frac{r}{2}\right), \frac{r}{3M}\right)  \subset B(p,r)\cap\Oj.
 \end{equation}
The fact that $\overline{\Oj} \to \overline{\Oinf}$, the first inclusion in \eqref{eqn109}, and \eqref{saegt} give for all $j$ large enough
 \begin{equation}\label{eq:temp107}
 	B\left(A(p,r),\frac{r}{4M}\right) \subset \overline{\Oinf}
 	\qquad\mbox{and}\qquad
 	\dist\left( B\left(A(p, r), \frac{r}{4M}\right), \pOj \right) \geq 
 	\frac{r}{6M}.
 \end{equation} 
This and the fact that $\pOj \to \pOinf$ yield that $\dist(B(A(p,r),r/4M), \pOinf)\ge r/6M$, hence $B(A(p,r),r/4M)$ misses $\pOinf$. Combining this with \eqref{eq:temp107} and the second inclusion in \eqref{eqn109}, we conclude that 
\begin{equation}\label{eq:NTOinf}
 			B\left(A(p,r),\frac{r}{4M} \right) \subset \Oinf \cap B(p,r).
 		\end{equation}
Hence, $\Omega_\infty$ satisfies the interior corkscrew condition with constant $4M$.
\medskip

\noindent \textbf{Harnack chain condition.} Fix $X,Y\in \Oinf$ and pick $q_X, q_Y\in\pOinf$ such that $|X-q_X| = \delta_{\infty}(X), |Y-q_Y|=\delta_\infty(Y)$. 	Without loss of generality we may assume that $\delta(X)\ge \delta(Y)$ (otherwise we switch the roles of $X$ and $Y$). Let us recall that every $\Omega_j$ satisfies the Harnack chain condition with constants $M, C_1>1$.
Set 
\begin{equation}\label{Def:Theta}
\Theta:=M\left(2+\log_2^+\left( \frac{|X-Y|}{\min\{\delta_{\infty}(X),\delta_{\infty}(Y)\}} \right)\right)
=
M\left(2+\log_2^+\left( \frac{|X-Y|}{\delta_{\infty}(Y)} \right)\right).
\end{equation}
  Choose $R\ge $ large enough (depending on $X,Y$) so that
\begin{equation}\label{BqX}
B(q_X,\delta_\infty(X)/2), B\big(X, (2C_1^2)^{4\Theta}\delta_\infty(X)\big)
\subset B(0,R)
\end{equation}
and
\begin{equation}\label{BqY}
B(q_Y,\delta_\infty(Y)/2),  B\big(Y, (2C_1^2)^{4\Theta}\delta_\infty(Y)\big)\subset B(0,R)
\end{equation}
Take also $d=2^{-1} C_1^{-2\Theta}\le 1$ which also depends on $X,Y$. Then, by \ref{3-thm:pseudo-blow-geo} in Theorem \ref{thm:pseudo-blow-geo} we can take $j$ large enough (depending on $R$ and $d$) so that
	\begin{equation}\label{eq:temp109}
	    D\big[\pOj \cap \overline{B(0,R)}, \pOinf\cap \overline{B(0,R)}\big], D\big[\overline\Oj \cap \overline{B(0,R)}, \overline\Oinf\cap \overline{B(0,R)}\big] \leq  \frac{d}{2} \delta_{\infty}(Y)\le \frac{d}{2}  \delta_{\infty}(X),
	\end{equation}
By \eqref{eq:temp109}, \eqref{BqX}, and \eqref{BqY} we have that $X,Y\in \Oj$, and 
	\begin{equation}\label{eqn110}
	 \frac{\delta_{\infty}(X)}{2} \leq \wdj(X) \leq \frac{3\delta_{\infty}(X)}{2}\qquad\mbox{and}\qquad\frac{\delta_{\infty}(Y)}{2} \leq \wdj(Y) \leq \frac{3\delta_{\infty}(Y)}{2}. 
	 \end{equation}
	Since $\Oj$ satisfies the Harnack chain condition with constants $M, C_1>1$, there exists a collection  of balls $B_1,\dots, B_K$ (the choice of balls  depend on the fixed $j$) connecting $X$ to $Y$ in $\Oj$ and such that
	\begin{equation} \label{eq:HBj}
		C_1^{-1} \dist(B_k,\partial\Omega_j) \leq \diam(B_k) \leq C_1 \dist(B_k,\partial\Omega_j),  
	\end{equation} 
	for $k=1,2,\dots,K$ where
	\begin{equation}\label{eqn111}
	 K 
	 \leq 
	 M\left(2+\log_2^+\left( \frac{|X-Y|}{\min \{ \wdj(X), \wdj(Y) \}} \right)\right)
	 \leq 
	 2\Theta.
	  \end{equation}
	  Combining \eqref{eq:HBj} and \eqref{eqn111}, one can see that for every $k = 1,2,\dots,K$
	  \begin{equation}\label{aaa1}
	  \dist(B_k,\partial\Omega_j)  \geq d \delta_{\infty}(X), \qquad \diam(B_k)\le (2C_1^2)^{2\Theta}\delta_\infty(Y)
	  \end{equation}
	  and
	  \begin{equation}\label{aaa2}
\dist(X,B_k)\le 2(2C_1^2)^{2\Theta}\delta_{\infty}(X),\qquad \dist(Y,B_k) \le 2(2C_1^2)^{2\Theta} \delta_{\infty}(Y).
	  \end{equation}

Given an arbitrary $q_j\in\partial\Omega_j\setminus \overline{B(0,R)}$, by \eqref{BqX}, \eqref{aaa1},	and \eqref{aaa2}	 it follows that 
\begin{multline}\label{rfarefer}
(2C_1^2)^{4\Theta}\delta_\infty(X)\le |q_j-X|\le \dist(q_j,B_k)+\diam(B_k)+\dist(X,B_k)
\\
\le
\dist(q_j,B_k)+
3(2C_1^2)^{2\Theta}\delta_{\infty}(X).
\end{multline}
Hiding the last term, using that $\Theta>2$ and taking the infimum over the $q_j$ as above we conclude that
\begin{equation}\label{gvasf}
4C_1 (2C_1^2)^{2\Theta}\delta_\infty(X)< \dist(B_k,\partial\Omega_j\setminus \overline{B(0,R)}).
\end{equation}
On the other hand, by \eqref{eq:HBj} and \eqref{aaa1}
$$
\dist(B_k,\partial\Omega_j)\le C_1\diam(B_k)
\le C_1(2C_1^2)^{2\Theta}\delta_\infty(Y)
\le
C_1(2C_1^2)^{2\Theta}\delta_\infty(X),
$$
which eventually leads to $\dist(B_k,\partial\Omega_j)=\dist(B_k,\partial\Omega_j\cap \overline{B(0,R)})$.
Analogously, replacing $q_j$ by $q\in\partial\Omega_\infty\setminus \overline{B(0,R)}$ in \eqref{rfarefer} we can easily obtain that \eqref{gvasf} also holds for $\Omega_\infty$: 
\begin{equation}\label{gvasserf}
4C_1 (2C_1^2)^{2\Theta}\delta_\infty(X)< \dist(B_k,\partial\Omega_\infty\setminus \overline{B(0,R)}).
\end{equation}
But, \eqref{aaa2} yields
$$
\dist(B_k,\partial\Omega_\infty)
\le
\delta_\infty(X)+\dist(X,B_k)
\le
\delta_\infty(X)+
 2(2C_1^2)^{2\Theta}\delta_\infty(Y)
 \le 3(2C_1^2)^{2\Theta}\delta_\infty(Y)
 ,
$$
which eventually leads to $\dist(B_k,\partial\Omega_\infty)=\dist(B_k,\partial\Omega_\infty\cap \overline{B(0,R)})$.
Using all these, \eqref{eq:temp109}, the triangular inequality and \eqref{eq:temp109} we can obtain
\begin{multline*}
\big|\dist(B_k,\partial\Omega_j) -\dist(B_k,\partial\Omega_\infty)\big| 
=
\big|\dist(B_k,\partial\Omega_j\cap \overline{B(0,R)})-\dist(B_k,\partial\Omega_\infty\cap \overline{B(0,R)})\big|
\\
\leq D\big[\pOj \cap \overline{B(0,R)},\pOinf\cap \overline{B(0,R)}\big] \leq \frac{d}{2} \delta_{\infty}(X) 
\leq \frac12 \dist(B_k,\Omega_j).
\end{multline*}
Thus,
\begin{equation}\label{y6e5}
\frac23 \dist(B_k,\partial\Omega_\infty) \leq \dist(B_k,\partial\Omega_j) \leq 2 \dist(B_k,\partial\Omega_\infty).
\end{equation}
and moreover $B_k\cap\partial\Omega_\infty=\emptyset$. Note that the latter happens for all $k=1,\dots, K$. Recall also that $X\in B_1\cap\Omega_\infty$ and that $B_k\cap B_{k+1}\neq\emptyset$. Consequently, we necessarily have that $B_k \subset \Oinf$ for all $k=1,\dots, K$. Furthermore, \eqref{y6e5} and \eqref{eq:HBj} give
	  \begin{equation} \label{eq:HBj-infty}
		    \frac{2}{3} C_1^{-1} \dist(B_k,\partial\Omega_\infty)\leq \diam(B_k) \leq 2C_1 \dist(B_k,\partial\Omega_\infty). 
	\end{equation}
To summarize, we have found a chain of balls $B_1, \dots, B_K$, all contained in $\Oinf$, which  verify \eqref{eq:HBj-infty}, and connect $X$ to $Y$. Also,
$K$ satisfies \eqref{eqn111} with $\Theta$ given in \eqref{Def:Theta}. Therefore $\Omega_\infty$ satisfies the Harnack chain condition with constants $2M$ and $2C_1$. This completes the proof of  \ref{4-thm:pseudo-blow-geo} in Theorem \ref{thm:pseudo-blow-geo}.
\end{proof}

\begin{proof}[Proof of \ref{5-thm:pseudo-blow-geo} in Theorem \ref{thm:pseudo-blow-geo}]
We first recall that for every $j$, $\sj=\HH^{n-1}|_{\partial\Omega_j}$ is an Ahlfors regular measure with constant $C_{AR}$ and hence $\spt\sj=\partial\Omega_j$. 
In particular the sequence $\{\sj\}$ satisfies conditions \ref{1-lm:sptcv} and \ref{2-lm:sptcv} of Lemma \ref{lm:sptcv}. 

 On the other hand, the fact that $\partial\Omega_j$ is Ahlfors regular easily yields, via a standard covering argument, that $\HH^{n-1}(\partial\Omega_j) \le 2^{n-1}C_{AR} \diam(\Omega_j)^n$. Hence, using again that $\partial\Omega_j$ is Ahlfors regular we conclude that for every $R>0$
\[ 
\sup_j \sj(B(0,R)) = \sup_j \HH^{n-1}(\partial\Omega_j \cap B(0,R)) \leq 2^{n-1}C_{AR}  R^{n-1}.
\]
Therefore modulo passing to a subsequence (which we relabel), there exists a Radon measure $\sinf$ such that $\sj \rightharpoonup \sinf$ as $j\to\infty$.
Using Lemma \ref{lm:sptcv}, $\partial\Omega_j=\spt\sj \to \spt\sinf$  as $j\to\infty$ in the sense of Definition \ref{def:cvsets}. This and  \ref{3-thm:pseudo-blow-geo} in Theorem \ref{thm:pseudo-blow-geo}  lead to $\spt\sinf = \pOinf$. 

 To show that $\mu_\infty$ is Ahlfors regular take $q\in\partial\Omega_\infty$. Let $q_j\in\pOj$ be such that $q_j\to q$ as $j\to\infty$. For  any $r>0$, using
 \cite[Theorem 1.24]{Ma}  and that $\sigma_j$ is Ahlfors regular with constant $C_{AR}$ we conclude that
	\begin{equation}\label{eq:upperAR}
\mu_\infty(B(q,r)) \leq \liminf_{j\to\infty}  \sj(B(q,r)) \leq \liminf_{j\to\infty} \sj(B(q_j,2r)) 
\leq 2^{n-1} C_{AR}r^{n-1}.
\end{equation}
On the other hand, let $0<r<\diam(\partial\Omega_\infty)$. In \ref{CaseII}, by \eqref{diampOinf} we get that $r<\diam(\pOj)$ for all $j$ sufficiently large. In \ref{CaseI}, either $\diam(\pOinf) = \infty$ or $\diam(\pOinf) < \infty$, but we still have $r<\diam(\pOj)$ for all $j$ sufficiently large. Hence, using again \cite[Theorem 1.24]{Ma} and that $\sigma_j$ is Ahlfors regular with constant $C_{AR}$ we obtain 
  	\begin{multline} \label{eq:lowerAR}
		\mu_\infty(B(q,r)) \geq \mu_\infty\left(\overline{B\left(q,\frac{r}{2}\right)}\right)  
		\geq 
		\limsup_{j\to\infty} \sj\left(\overline{B\left(q,\frac{r}{2}\right)}\right) 
		\\
		\geq 
		\limsup_{j\to\infty} \sj\left(B\left(q_j,\frac{r}{4}\right) \right) 
		 \geq 4^{-(n-1)}C_{R}^{-1}r^{n-1}. 
	\end{multline} 
 These estimates guarantee that $\minf$ is Ahlfors regular with constant $2^{2(n-1)}C_{AR}$. Moreover by \cite[Theorem 6.9]{Ma}, 
	\begin{equation}  \label{eqn115}
	2^{-2(n-1)}C_{AR}^{-1} \minf \leq \mathcal{H}^{n-1}|_{\pOinf} \leq 2^{3(n-1)}C_{AR} \minf. 
	\end{equation}
	and consequently $\partial\Omega_\infty$ is Ahlfors regular with constant $2^{5(n-1)}C_{AR}^2$. This completes the proof of \ref{5-thm:pseudo-blow-geo} and hence that of Theorem \ref{thm:pseudo-blow-geo}. 
\end{proof}

\subsection{Convergence of elliptic matrices}
Our next goal is to show that there exists a constant coefficient real 
symmetric elliptic matrix $\wcalA^*$ with 
ellipticity constants 
$1=\lambda\le \Lambda<\infty$
(as in \eqref{def:UE}) so that for any $0<R<\diam(\pOinf)$ and for any $1\le p<\infty$.
\begin{equation}\label{eq:temp118}
\int_{ B(0,R) \cap \Oj } |\wcalA_{j}(Z) - \wcalA^*|^p dZ \to 0,\qquad \text{ as } j\to \infty.
\end{equation}

Fix $Z_0 \in \Oinf$ and set $B_0 = B(Z_0, 3\delta_{\infty}(Z_0)/8)$. Since $\pOj \to \pOinf$ and $\overline\Oj \to \overline\Oinf$ as $j\to\infty$,  for all sufficiently large $j$, we can see that $Z_0 \in\Oj$, 
\begin{equation}\label{eq:djdinfty}
	\frac{3}{4} \delta_{\infty}(Z_0) \leq  \wdj(Z_0) \leq \frac{5}{4} \delta_{\infty}(Z_0),
\end{equation}  
and 
\begin{equation}
B_0 \subset B\left(Z_0,\frac{\wdj(Z_0)}{2} \right) \subset \frac{5}{3} B_0 \subset \Oj \quad \text{for all } j.
\end{equation}
All these, Poincar\'e's inequality, and \eqref{def:smallCarleson} yield
\begin{multline}\label{esti-conve-Aj}
\fint_{B_0} |\wcalA_{j}(Z) - \langle\wcalA_{j}\rangle_{B_0}| dZ 
\lesssim 
 \delta_{\infty}(Z_0)
\fint_{B_0} |\nabla \wcalA_{j}(Z)| dZ 
\\
\lesssim
\fint_{B\left(Z_0, \wdj(Z_0)/2 \right)} |\nabla \wcalA_{j}| \wdj(Z) dZ
\le \mathcal{C}(\Oj,\wcalA_j) 
<\epsilon_j.
\end{multline}

\begin{remark}\label{rem:oscila-KP} 
We note that if we state the Main Theorem using the oscillation assumption \eqref{def:oscA}, we can easily conclude the same estimate: 
\[
	\fint_{B_0} |\wcalA_{j}(Z) - \langle\wcalA_j\rangle_{B_0}| dZ
	\lesssim
	\fint_{B(Z_0,\delta_j(Z_0)/2)} |\mathcal{A}_j(Z) - \langle\mathcal{A}_j\rangle_{B(Z_0,\delta_j(Z_0)/2)}| dZ 
	\le
	\osc(\Omega_j, \mathcal{A}_j)
		< \epsilon_j.
\]
From here the proof continues the same way.
\end{remark}

Note that all the matrices $\wcalA_{j}$ are 
uniformly elliptic (i.e., all of them satisfy \eqref{def:UE}), with the 
same constants 
$1=\lambda\le\Lambda<\infty$,
and in particular $\{\langle\wcalA_{j} \rangle_{B_0}\}_j$ is a bounded sequence of constant real matrices. Hence, passing to a subsequence and relabeling $\langle\wcalA_{j} \rangle_{B_0}$ converges to some constant elliptic matrix, denoted by $\wcalA^*(B_0)$. Combining this with \eqref{esti-conve-Aj}, the dominated convergence theorem yields 
\begin{equation}\label{conv-B0}
	\fint_{B_0} |\wcalA_{j} (Z) - \wcalA^*(B_0)| dZ \to 0\quad \text{ as } j\to\infty,
\end{equation}
that is,  $\wcalA_{j}$ converges in $L^1(B_0)$ to a constant elliptic matrix $\wcalA^*(B_0)$. Moreover, passing to a further subsequence an relabeling $\wcalA_{j}\to \wcalA^*(B_0)$ almost everywhere in $B_0$. In 
particular, $\wcalA^*(B_0)$ is a real symmetric elliptic matrix
(i.e., it satisfies \eqref{def:UE}), with ellipticity constants 
$1=\lambda\le \Lambda<\infty$. 
It is important to highlight that all the previous subsequences and relabeling only depends 
on the choice of $Z_0 \in \Oinf$.  In any case, since $\wcalA^*(B_0)$ is a constant coefficient matrix we set 
$\wcalA^*:=\wcalA^*(B_0)$.

Let us pick a countable collection of points $\{Z_k\}\subset \Omega_\infty$ so that $\Omega_\infty=\cup_k B_k$ with $B_k = B(Z_k, 3\delta_{\infty}(Z_k)/8)$. 
We can repeat the previous argument with any $Z_k$ and define $\wcalA^*(B_k)$, a constant  real symmetric elliptic matrix satisfying \eqref{def:UE} so that for some subsequence depending on $k$, we obtain that  $\wcalA_{j}\to \wcalA^*(B_k)$ in $L^1(B_k)$ and a.e in $B_k$ as $j\to\infty$. In particular, $\wcalA^*(B_{k_1})=\wcalA^*(B_{k_2})$ a.e. in  $B_{k_1}\cap B_{k_2}$ (in case it is non-empty). Note that $\Oinf$ is path connected (since it satisfies the Harnack chain condition), hence for any $k$ we can find a path joining $Z_k$ and $Z_0$ and cover this path with a finite collection of the previous balls to easily see that $\wcalA^*(B_{k})=\wcalA^*=\wcalA^*(B_0)$. Moreover, using a diagonalization argument, we can show that there exists a subsequence, which we relabel, so that for all $k$, we have that  $\wcalA_{j}\to \wcalA^*$ in $L^1(B_k)$ and a.e in $B_k$ as $j\to\infty$. From this, and since the matrices concerned are all uniformly bounded, one can prove that for any $1\leq p<\infty$ and for all $Z\in\Omega_\infty$
\begin{equation}\label{Ajcvinp}
\fint_{B_Z} |\wcalA_{j}(Y) - \wcalA^*|^p dY \to 0\quad \text{ as } j\to\infty,
\end{equation}
where $B_Z=B(Z,\delta(Z)/2)$.

We are now ready to start proving our claim \eqref{eq:temp118}. Recalling that $\overline{\Oj} \to \overline{\Oinf}$, $\pOj \to \pOinf$ in the sense of Definition \ref{def:cvsets}, and that $\pOj, \pOinf$ have zero Lebesgue measure since they are Ahlfors regular sets, one can see that 
\begin{align}
	B(0,R) \cap \left( \Oj \triangle \Oinf \right) \subset B(0,R) \cap \left( \left(\overline\Oj \triangle \overline\Oinf \right) \cup \left(\overline\Oj \cap \pOinf \right) \cup \left( \overline\Oinf \cap \pOj \right) \right)
\end{align}
and hence the Lebesgue measure of the set on the left hand side  tends to zero as $j\to\infty$. This and the fact that $\|\wcalA_{j}\|_\infty, \|\wcalA^*\|_\infty\le\Lambda$ give
\begin{equation}\label{eq:temp119}
	\int_{B(0,R) \cap \left(\Oj \triangle \Oinf \right)} |\wcalA_{j}(Z) - \wcalA^*|^p dZ \to 0,\qquad \text{ as } j \to \infty.
\end{equation}

On the other hand, let $\varrho>0$ be arbitrarily small and let $\epsilon = \epsilon(\varrho)>0$ be a small constant to be determined later. Set 
\[
\Oinf^{\epsilon,1} := B(0,R) \cap \{Z\in\Oinf: \delta_{\infty}(Z) < \epsilon\}
\quad\mbox{and}\quad
\Oinf^{\epsilon,2} := B(0,R) \cap \{Z\in\Oinf: \delta_{\infty}(Z) \geq \epsilon\}.
\]
Using the notation $\Delta(q,r):=B(q,r)\cap \pOinf$ with $q\in\pOinf$ and $r>0$, Vitali's covering lemma allows us to find a finite collection of balls $B(q_i,\epsilon)$ with $q_i\in \Delta(0,R+\epsilon)$, such that
\begin{equation}\label{est:Omega-near}
\Oinf^{\epsilon,1} \subset \bigcup_{i} B(q_i,5\epsilon). 
\end{equation}
Calling the number of balls $L_1$ we get the following estimate 
\begin{align}\label{eq:countL2}
L_1\epsilon^{n-1} 
\lesssim
\sum_{i} \sigma_\infty \left(  \Delta(q_i,\epsilon) \right) 
= 
\sigma_\infty \Big( \bigcup_i\Delta(q_i,\epsilon) \Big) \leq \sigma_\infty \left( \Delta(0,R+2\epsilon) \right)
\lesssim 
(R+2\epsilon)^{n-1},
\end{align}
where we have used that $\pOinf$ is Ahlfors regular and also that $\Delta(q_i, \epsilon) \subset \Delta(0, R+2\epsilon)$ since $q_i \in \Delta(0,R+\epsilon)$. If we assume that $0<\epsilon<R$ we conclude that $L_1\lesssim (R/\epsilon)^{n-1}$ and moreover by \eqref{est:Omega-near} we conclude that $|\Oinf^{\epsilon,1}|\lesssim \epsilon$ (here the implicit constant depend on $R$). This and $\|\wcalA_{j}\|_\infty, \|\wcalA^*\|_\infty\le\Lambda$ give at once that for every $j$
\begin{equation}\label{eq:e1}
	\int_{\Oinf^{\epsilon,1} \cap \Oj} |\wcalA_{j}(Z) - \wcalA^*|^p dZ \lesssim \Lambda^p \epsilon<\frac{\varrho}2, 
\end{equation} 
provided $\epsilon$ is taken small enough which is fixed from now on. 

On the other hand, note that $\overline{\Oinf^{\epsilon,2}}$ is compact, hence we can find $Z_1,\dots, Z_{L_2}\in \overline{\Oinf^{\epsilon,2}}$ so that 
$\Oinf^{\epsilon,2} \subset \bigcup_{i=1}^{L_2} B_{Z_i}$ where $L_2$ depends on $\epsilon$ and $R$ which have been fixed already. Hence, by \eqref{Ajcvinp}
\[
	\int_{\Oinf^{\epsilon,2} \cap \Oj } |\wcalA_{j}(Z) - \wcalA^*|^p dZ 
	\leq \sum_{i=1}^{L_2} \int_{B_{Z_i}} |\wcalA_{j}(Z) - \wcalA^*|^p dZ 
	\to 0,\quad \text{ as } j\to\infty.
\]
In particular, we can find an integer $j_0 = j_0(R,\epsilon)$ such that
\begin{equation}\label{eq:e2}
	\int_{\Oinf^{\epsilon,2} \cap \Oj } |\wcalA_{j}(Z) - \wcalA^*|^p dZ < \frac{\varrho}{2}, \quad \text{ for any } j\geq j_0.
\end{equation}
Combining \eqref{eq:e1} and \eqref{eq:e2}, we conclude that
\begin{equation}\label{eq:temp120}
	\int_{B(0,R) \cap \left(\Oj \cap\Oinf\right)} |\wcalA_{j}(Z) - \wcalA^*|^p dZ <\varrho, \quad \text{ for any } j\geq j_0.
\end{equation}
This combined with \eqref{eq:temp119} proves the claim \eqref{eq:temp118}.

\subsection{Convergence of operator}
\begin{theorem}\label{thm:AW11}
The function $u_{\infty}$ solves the Dirichlet problem 
\begin{equation}\label{limit-eq2}
	 \left\{ \begin{array}{rl}
		-\divg(\mathcal{A}^* \nabla u_\infty) = 0 & \text{in } \Oinf, \\
		u_\infty > 0 & \text{in }\Oinf, \\
		u_\infty = 0 & \text{on } \pOinf,
	\end{array} \right.  
\end{equation}
in \ref{CaseI}, and solves the Dirichlet problem
\begin{equation}\label{limit-eq2case2}
	 \left\{ \begin{array}{rl}
		-\divg(\mathcal{A}^* \nabla u_\infty) = \delta_{\{X_0\}} & \text{in } \Oinf, \\
		u_\infty > 0 & \text{in }\Oinf, \\
		u_\infty = 0 & \text{on } \pOinf,
	\end{array} \right.
\end{equation}
in \ref{CaseII}. Hence, $u_\infty$ is a Green function in $\Oinf$ for a constant-coefficient elliptic operator $L_\infty=-\divg(\mathcal{A}^* \nabla)$ with pole at $\infty$ in \ref{CaseI} or at $X_0\in\Oinf$ in \ref{CaseII}.
\end{theorem}

\begin{proof}
Let $\psi\in C^\infty_c(\Oinf)$. Since $\overline{\Oj}\to\overline{\Oinf}$ and $\pOj\to\pOinf$, it follows that $\psi\in C^\infty_c(\Oj)$  for $j$ sufficiently large. 
In \ref{CaseI}, using \eqref{def:ujcasei} and \eqref{eqn:int-parts}  we have
	\begin{equation}\label{eqn116}
		\int_{\Oj} \langle \wcalA_{j}\nabla \uj,\nabla \psi\rangle dZ 
		=
	 \frac{1}{\oj^{X_j}(B(0,1)) }\int_{\Omega_j} \langle\wcalA_j\nabla G_j(X_j,\cdot), \nabla \psi\rangle  dZ
	 =  
	 \frac{\psi(X_j)} {\omega_j^{X_j}(B(0,1))} \to 0,
	 \end{equation}
	 as $j\to\infty$ since $\wXj \to \infty$ by \eqref{eq:temp1}. 
	 Analogously, in \ref{CaseII}, by \eqref{def:ujcaseii} and \eqref{eqn:int-parts} we obtain 
	 \begin{equation}\label{eqn116case2}
	\int_{\Oj} \langle \wcalA_{j}\nabla \uj,\nabla \psi\rangle dZ 
	=
	\int_{\Omega_j} \langle\wcalA_j\nabla G_j(X_j,\cdot), \nabla \psi\rangle  dZ
	=
	\psi(X_j) \to \psi(X_0).
	 \end{equation}
	 as $j\to\infty$ since $\wXj \to X_0$.
	
	Suppose next that $\spt \psi \subset B(0,R)$. Let $r=2$ for \ref{CaseI}, and pick $r\in [1,n/(n-1))$ for \ref{CaseII}. 	
	By \ref{1-thm:pseudo-blow-geo} in Theorem \ref{thm:pseudo-blow-geo} in \ref{CaseI} and \ref{3-rem:deawffr} in Remark \ref{rem:deawffr} in \ref{CaseII} it follows that $\nabla \uj \rightharpoonup \nabla u_{\infty}$ in $L^r(B(0,R))$. On the other hand, 
	\begin{align}\label{eqn:3.8}
		& \left| \int_{\Oj}  \langle \wcalA_{j}\nabla \uj,\nabla \psi\rangle dZ - \int_{\Oinf}  \langle \wcalA^*\nabla u_\infty,\nabla \psi\rangle
		dZ \right|  
		\\
				& \qquad \leq \|\nabla \psi\|_{L^{\infty}} \left( \int_{\Oj\cap B(0,R)} |\wcalA_{j} - \wcalA^* |^{r'} dZ \right)^{\frac{1}{r'}} \left( \int_{\Oj\cap B(0,R)} |\nabla \uj|^r \right)^{\frac{1}{r}} 
		\nonumber \\
		& \qquad \qquad \qquad + \left| \int_{\Oj \cap B(0,R)} \langle \wcalA^* \nabla \uj, \nabla \psi \rangle dZ - \int_{\Oinf \cap B(0,R)} \langle \wcalA^* \nabla u_{\infty}, \nabla \psi \rangle dZ \right|. \nonumber 
	\end{align}
	Using \eqref{eqn:3.4} in \ref{CaseI} or \eqref{eq:upperr} in \ref{CaseII},  and \eqref{eq:temp118} with $p=r'$, the term in the second line of \eqref{eqn:3.8} tends to zero as $j\to\infty$.
	Concerning the last term, since $\wcalA^*$ is a constant-coefficient matrix, it follows that $\wcalA^* \nabla \uj \rightharpoonup \wcalA^* \nabla u_{\infty}$ in $L^r(B(0,R))$. Moreover $\overline\Oj=\overline{\{\uj > 0 \}} \to \overline\Oinf=\overline{\{u_{\infty} > 0 \}}$, thus
\[
		\lim_{j\to\infty} \int_{\Oj} \langle \wcalA^* \nabla \uj, \nabla \psi\rangle = \int_{\Oinf} \langle \wcalA^* \nabla u_{\infty}, \nabla \psi \rangle.
\]
	Combining these with \eqref{eqn116}--\eqref{eqn:3.8} we eventually conclude that
	\begin{equation}
		\int_{\Oinf} \wcalA^* \nabla u_{\infty} \cdot \nabla \psi = 0 \quad \text{for all } \psi \in C_c^\infty(\Oinf)
	\end{equation}
	in \ref{CaseI}, i.e., $-\divg(\mathcal{A}^*\nabla u_{\infty}) = 0$ in $\Oinf$; and in \ref{CaseII},
	\begin{equation}
		\int_{\Oinf} \wcalA^* \nabla u_{\infty} \cdot \nabla \psi = \psi(X_0) \quad \text{for all } \psi \in C_c^\infty(\Oinf),
	\end{equation}
	 i.e., $-\divg(\mathcal{A}^*\nabla u_{\infty}) = \delta_{\{X_0\}}$ in $\Oinf$.
\end{proof}

\subsection{Analytic properties of the limiting domains}
As mentioned in Section \ref{sect:blowup}, in order to apply Theorem \ref{thm:hmu} we need to study the elliptic measures of the limiting domain with finite poles. In this section we construct these measures by a limiting procedure which is compatible with the procedure used to produce the limiting domain $\Oinf$.

\begin{theorem}\label{thm:blow-ana-pole}
	Under  \ref{1-assump},  \ref{2-assump},  \ref{3-assump}, and using the notation from Theorems \ref{thm:pseudo-blow-geo} and \ref{thm:AW11}, 
	the elliptic measure $\omega_{L_{\infty}}\in A_\infty(\sigma_\infty)$ (see Definition \ref{def:AinftyHMU}) with constants $\widetilde{C}_0=C_2 C_{AR}^{4\theta} 2^{8(n-1)\theta}$ and $\widetilde{\theta}=\theta$, here $C_2$ is the constant in Remark \ref{rem:doubling:needed}.
\end{theorem}

\begin{proof}
	Our goal is to show that the elliptic measures of $L_{\infty}$ with finite poles can be recovered as a limit of the elliptic measures of $\wLj = -\divg(\wcalA_{j}(Z)\nabla)$, and the $A_{\infty}$ property of elliptic measures is preserved when passing to a limit.

To set the stage we start with $0\le f\in \Lip(\partial\Omega_\infty)$ with compact support. Let $R_0>0$ be large enough so that $\spt f\subset B(0,R_0/2)$. We are going to take a particular solution to the following Dirichlet problem 
	\begin{equation}\label{D:Oinf}
\left\{ \begin{array}{ll}
L_\infty v = 0, & \text{in }\Oinf \\
v= f, & \text{on }\pOinf,
\end{array} \right.
\end{equation} 

In \ref{CaseII}, where the domain $\Oinf$ is bounded, the Dirichlet problem \eqref{D:Oinf} has a unique solution satisfying the maximum principle,
then we let $v_\infty$ be that unique solution. 

In \ref{CaseI}, where $\Oinf$ is unbounded, we follows the construction in \cite{HM1} using Perron's method (see \cite[pg. 588]{HM1} for details, which is done the Laplacian but holds for any constant coefficient operator, for the general case see also \cite{HMT2}). We denote the solution constructed in \cite{HM1} by 
\[
v_\infty(Z)=
\int_{\pOinf} f(q) d\omega_{L_\infty}^Z(q).
\]
For later use we need to sketch how it is constructed. For every $R>4R_0$ define $f_R=f \eta(\cdot/R)$, where $\eta\in C_c^\infty(B(0,2R)$ verifies $0\leq \eta\leq 1$, $\eta = 1$ for $|Z| < 1$. Let $v_R$ be the unique solution to $L_{\infty} v_R = 0$ in the bounded open set $\Omega_R = \Oinf \cap B(0,2R)$ with boundary value $f_R$.  Then one shows that $v_R\to v_\infty$ uniformly on compacta as $R\to\infty$,  and also that $v_\infty\in C(\overline{\Omega_\infty})$ satisfies the maximum principle $0\le \max_{\Omega_\infty}v_\infty\le \max_{\partial\Omega_\infty} f$.

Once the solution $v_\infty$ is defined we observe that since $\partial\Omega_\infty$ is Ahlfors regular we can use the 
Jonsson-Wallin trace/extension theory \cite{JW} to extend $f$ (abusing the notation we call the extension $f$) so that $0\le f\in C_c(\R^n)\cap W^{1,2}(\R^n)$ with $\spt f\subset B(0,R_0)$. For every $j$ we let $h_j\in W^{1,2}_0(\Omega_j)$ be the unique Lax-Milgram solution to the problem $L_j h_j= L_j f$. Initially, $h_j$ is only defined in $\Omega_j$ but we can clearly extend it by $0$ outside so that the resulting function, which we call again  $h_j$, belongs to $W^{1,2}(\R^n)$. If we next set $v_j=f-h_j\in W^{1,2}(\R^n)$ we obtain that $L_jv_j=0$ in $\Omega_j$ and indeed 
	\begin{equation}\label{eqn:102A}
		v_j(Z) = \int_{\pOj} f d\oj^Z, \qquad Z\in\Oj,
	\end{equation}
	see \cite{HMT2}. Here $\oj^Z$ is the elliptic measure of $\wLj$ in $\Oj$ with pole $Z$ and, as observed above, the fact that $\partial\Omega_j$ is Ahlfors regular implies in particular that $v_j\in C(\overline{\Omega_j})$ with $v_j|_{\partial\Omega_j}=f$. Note also that $v_j=f\in C(\R^n)$ on $\R^n\setminus\Omega_j$, hence $v_j\in\ C(\R^n)$. Moreover, by the maximum principle 
\begin{equation}\label{eqn:max-pple}
0\le \sup_{\Oj} v_j \leq \|f\|_{L^{\infty}(\pOj)} \leq \|f\|_{L^{\infty}(\RR^n)},
\end{equation}
thus the sequence $\{v_j\}$ is uniformly bounded. 
	
Our next goal is to show that $\{v_j\}$ is equicontinuous. Given an arbitrary $\varrho>0$ let $0<\gamma<\frac1{32}$ to be chosen. Since $f\in C_c(\R^n)$, it is uniformly continuous, hence letting $\gamma$ small enough (depending on $f$) we can guarantee that
	\begin{equation}\label{f-uc}
		|f(X)-f(Y)|<\frac{\varrho}8,
		\qquad
		\mbox{provided \,} |X-Y|<\gamma^{\frac14}
	\end{equation}

Our first claim is that if $\gamma$ is small enough depending on $n$, $C_{AR}$, $\Lambda$ (recall that
$\lambda = 1$), and $\|f\|_{L^{\infty}(\R^n)}$, there holds 
\begin{equation}\label{claim-equicont}
|v_j(X)-v_j(Y)|<\frac{\rho}{2},
\qquad
\forall\,X\in\Omega_j,\ Y\in\partial\Omega_j,\ |X-Y|<\sqrt{\gamma}.
\end{equation}
To see this we recall that $\pOj$ is Ahlfors regular with a uniform constant (independent of $j$), it satisfies the CDC with a uniform constant and \cite[Theorem 6.18]{HKM} (see also \cite{HMT2}) yields that for some $\beta>0$ and $C$ depending on $n$, $C_{AR}$, and $\Lambda$, but independent of $j$ (indeed this is the same $\beta$ as in Lemma \ref{lem:vanishing}), the following estimate holds:
\[
	\underset{B(Y_j,\sqrt{\gamma})\cap \Oj }{\osc} v_j
	\le 
	\underset{B(Y_j,\gamma^{1/4})\cap \pOj }{\osc} f  
	+ 
	C\|f\|_{L^{\infty}(\R^n)} \eta^{\beta}<\frac\varrho2,
\]
	where in the last estimate we have used \eqref{f-uc} and $\gamma$ has been chosen small enough so that $C\|f\|_{L^{\infty}(\R^n)} \eta^{\beta}<\varrho/4$.

We now fix $X,
 Y\in\R^n$ so that $|X-Y|<\gamma$ and consider several cases. 

\noindent\textbf{Case 1}: $X,Y\in\Omega_j$ with $\max\{\delta_j(X),\delta_j(Y)\}<\sqrt{\gamma}/2$.

In this case, we take $\widehat{x}\in\partial\Omega_j$ so that $|X-\hat{x}|=\delta_j(X)$. Note that $|Y-\widehat{x}|<\sqrt{\gamma}$ and we can use \eqref{claim-equicont} to obtain
\[
|v_j(X)-v_j(Y)|
\le
|v_j(X)-v_j(\widehat{x})|
+
|v_j(\widehat{x})-v_j(Y)|<\rho.
\]

\noindent\textbf{Case 2}: $X,Y\in \Omega_j$ with $\max\{\delta_j(X),\delta_j(Y)\}\ge \sqrt{\gamma}/2$.

Assuming without loss of generality that $\delta_j(X)\ge \sqrt{\gamma}/2$, necessarily $Y\in B(X,\delta_j(X)/2)\subset\Omega_j$. Then, by the interior H\"older regularity of $v_j$ in $\Oj$ (here $\alpha$ and $C$  depend only on $\Lambda$ and are independent of $j$) we conclude that
\[
|v_j(X) - v_j(Y)| 
\le 
C\left( \frac{|X-Y|}{\wdj(X)} \right)^{\alpha} \|v_j\|_{L^{\infty}(\Oj)} 
\le
C 2^\alpha \gamma^{\frac{\alpha}2}\|f\|_{L^{\infty}(\R^n)}
<
\varrho,
\]
provided $\varrho$ is taken small enough (again independently of $j$).

\noindent\textbf{Case 3}: $X,Y\notin \Omega_j$.

Here we just need to use \eqref{f-uc} and the fact that $v_j=f$ on $\R^n\setminus \Omega_j$:
\[
|v_j(X) - v_j(Y)| 
=
|f(X)-f(Y)|<
\rho.
\]

\noindent\textbf{Case 4}: $X\in \Omega_j$ and $Y\notin \Omega_j$.

Pick $Z\in\partial\Omega_j$ in the line segment joining $X$ and $Y$ (if $Y\in\partial\Omega_j$ we just take $Z=Y$) so that $|X-Z|,|Y-Z|\le |X-Y|<\gamma$. Using \eqref{claim-equicont}, the fact that $v_j=f$ on $\R^n\setminus \Omega_j$,  and \eqref{f-uc} we obtain
\[
|v_j(X) - v_j(Y)| 
\le
|v_j(X) - v_j(Z)| 
+
|v_j(Z) - v_j(Y)| 
<
\frac{\varrho}2+|f(Z) - f(Y)| 
<\varrho.
\]
	
If we now put all the cases together we have shown that, as desired,  $\{v_j\}$ is equicontinuous.

On the other hand, recalling that $h_j\in W_0^{1,2}(\Omega_j)$ satisfies $L_j h_j=L_j f$ in the weak sense in 
$\Omega_j$, that $f\in W^{1,2}(\R^n)$, and that $\lambda =1$, we see that
\begin{multline*}
 \|\nabla h_j\|_{L^2(\Omega_j)}^2
\le
\int_{\Omega_j} \langle \wcalA_{j}\nabla h_j,\nabla h_j\rangle dX
=
\int_{\Omega_j} \langle \wcalA_{j}\nabla f,\nabla h_j\rangle dX
\le
\Lambda
\|\nabla f\|_{L^2(\Omega_j)}
\|\nabla h_j\|_{L^2(\Omega_j)}.
\end{multline*}
We next absorb the  last term, use that $v_j=f-h_j$ and that $h_j$ has been extended as $0$ outside of $\Omega_j$:
\[
\|\nabla v_j\|_{L^2(\R^n)}
\le
\|\nabla f\|_{L^2(\R^n)}
+
\|\nabla h_j\|_{L^2(\R^n)}
=
\|\nabla f\|_{L^2(\R^n)}
+
\|\nabla h_j\|_{L^2(\Omega_j)}
\le (1+\Lambda)
\|\nabla f\|_{L^2(\R^n)}.
\]
This along with \eqref{eqn:max-pple} yield
\begin{equation}\label{unfi-vj}
 \sup_j \|\nabla v_j \|_{L^2(\RR^n)} \leq  (1+\Lambda) \|\nabla f\|_{L^2(\RR^n)}, \quad \text{and } \sup_j \|v_j \|_{L^2(B(0,R))} \leq C_R. 
\end{equation}
We notice that all these estimates hold for the whole sequence and therefore, so it does for any subsequence.

Let us now fix an arbitrary subsequence $\{v_{j_k}\}_k$. By \eqref{unfi-vj} there are a further subsequence and $v\in C(\R^n)\cap W^{1,2}_{\rm loc}(\RR^n)$ with $\nabla v \in L^2(\RR^n)$, such that $v_{j_{k_l}} \to v$ uniformly on compact sets of $\RR^n$ (hence $v\ge 0$) and $\nabla v_{j_{k_l}}\rightharpoonup \nabla v$ in $L^2(\RR^n)$ as $l\to\infty$. Here it is important to emphasize that the choice of the subsequence may depend on the boundary data $f$ and the fixed subsequence, and the same happens with $v$ , and this could be problematic, later we will see that  this is not the case.

To proceed we next see that $v$ agrees with $f$ in $\partial\Omega_\infty$. Given $p\in \pOinf$, there exist $p_{j_{k_l}} \in \partial\Omega_{j_{k_l}}$ with $p_{j_{k_l}} \to p$ as $l\to\infty$. Using the continuity of $v$ and $f$ at $p$, the uniform convergence of $v_{j_{k_l}} $ to $v$ on  $\overline{B(p,1)}$  and the fact that $v_{j_{k_l}} =f$ on $\partial\Omega_{j_{k_l}} $, we have
\begin{multline*}\label{eqn:105A}
|v(p) - f(p)| 
\leq |v(p) - v(p_{j_{k_l}})| + |v(p_{j_{k_l}}) - v_{j_{k_l}}(p_{j_{k_l}})| +|f(p_{j_{k_l}} ) - f(p)| 
\\
\leq |v(p) - v( p_{j_{k_l}} )| + \|v-v_{j_{k_l}} \|_{L^{\infty}(\overline{B(p,1)})} + |f( p_{j_{k_l}}) - f(p)|
\to 0,
\quad\mbox{as }l\to\infty,
\end{multline*}
thus $v(p)=f(p)$ as desired.

Next, we claim the function $v$ solves the Dirichlet problem \eqref{D:Oinf}. We know that $v\in C(\R^n)$ with $v=f$ in $\partial\Omega_\infty$. Hence, we only need to show that $L_\infty v = 0$ in $\Oinf$. To this aim, let us  take $\psi \in C_c^1(\Oinf)$ and let $R>0$ be large enough so that $\spt \psi\subset B(0,R)$. Since $\overline{\Oj} \to \overline{\Oinf}$, for all $l$ large enough we have that $\psi \in C_c^1(\Omega_{j_{k_l}})$ in which case 
 	\begin{equation}\label{eqn:106A}
		\int_{\R^n} \langle \wcalA_{j_{k_l}}\nabla v_{j_{k_l}},\nabla \psi\rangle dZ= 0,
	\end{equation}
	since $L_{j_{k_l}} v_{j_{k_l}}=0$ in $\Omega_{j_{k_l}}$ in the weak sense. Then, by \eqref{unfi-vj} and the fact that $\spt \psi\subset \Omega_\infty\cap\Omega_{j_{k_l}}\cap B(0,R)$,
	\begin{align*}
	& \left| 	\int_{\R^n} \langle \wcalA^* \nabla v, \nabla \psi \rangle dZ \right|=	
	\left| \int_{\Omega_{j_k}}  \langle \wcalA_{j_k}\nabla v_{j_{k_l}},\nabla \psi\rangle dZ - \int_{\Oinf}  \langle \wcalA^*\nabla v,\nabla \psi\rangle
	dZ \right|  
	\\
	& \qquad \leq  (1+\Lambda)\|\nabla f\|_{L^2(\R^n)}\|\nabla \psi\|_{L^{\infty}} \left( \int_{\Oj\cap B(0,R)} |\wcalA_{j_{k_l}} - \wcalA^* |^2 dZ \right)^{\frac{1}{2}}
	\nonumber \\
	& \qquad \qquad \qquad + \left| \int_{\R^n} \langle \wcalA^* \nabla v_{j_{k_l}}, \nabla \psi \rangle dZ - \int_{\R^n}  \langle \wcalA^* \nabla v, \nabla \psi \rangle dZ \right|\to 0,\quad\mbox{as }l\to\infty,
	\end{align*}
where we have used \eqref{eq:temp118} with $p=2$ for the term in the second line, and the fact that since $\wcalA^*$ is a constant-coefficient matrix, it follows that $\wcalA^* \nabla v_{j_{k_l}} \rightharpoonup \wcalA^* \nabla v$ in $L^2(\R^n)$ as $l\to\infty$. This eventually shows that $L_\infty v=0$ in $\Omega_\infty$.

In \ref{CaseII} when the domain $\Oinf$ is bounded, the Dirichlet problem \eqref{D:Oinf} has a unique solution, and it satisfies the maximum principle, hence we must have that $v=v_\infty$. Therefore, we have  shown that given any subsequence $\{v_{j_k}\}_k$ there is a further subsequence $\{v_{j_{k_l}}\}_l$ so that $v_{j_{k_l}}\to v_\infty$ uniformly on compact sets of $\RR^n$ and $\nabla v_{j_{k_l}} \rightharpoonup \nabla v_\infty$ in $L^2(\RR^n)$ as $l\to\infty$. This eventually shows that entire sequence $\{v_j\}$ satisfies $v_j\to v_\infty$ uniformly on compact sets of $\RR^n$ and $\nabla v_{j} \rightharpoonup \nabla v_\infty$ in $L^2(\RR^n)$ as $j\to\infty$.

In \ref{CaseI} where the limiting domain $\Oinf$ is unbounded, we need more work to show the solution $v$ is indeed $v_\infty$. 
Recall that $f\in C_c(\R^n)$ with $\spt f\subset B(0,R_0)$. Given $\epsilon>0$, there is an integer $j_0 = j_0(\epsilon, R_0)\in \NN$ such that for $j\geq j_{0}$
and for any $p_j' \in \pOj \cap B(0,4R_0)$, there is $p' \in \pOinf \cap B(0,5R_0) $ close enough to $p_j'$ so that $|f(p') -f(p'_j)|<\epsilon$. Consequently, 
	\begin{equation}\label{eqn:106C}
	\sup_{\pOj} |f| = \sup_{\pOj \cap B(0,4R_0)} |f| \leq \sup_{\pOinf \cap  B(0,5R_0)} |f| + \epsilon = \sup_{\pOinf} |f| + \epsilon.  
	\end{equation}

For any $Z\in \Oinf$ there exists a sequence $Z_j \in \Oj$ such that $Z_j \to Z$ and $Z_j\in \overline{B(Z,\delta_\infty(Z)/2)}$ for all $j$ large enough. Since $v\in C(\R^n)$ it follows that for $j$ large enough $|v(Z)-v(Z_j)|<\epsilon$.
All these together with \eqref{eqn:max-pple}	and the fact that $v_{j_{k_l}} \to v$ uniformly on compact sets of $\RR^n$ as $l\to\infty$ give that for all $l$ large enough 
	\begin{equation}\label{eqn:106D}
0\le v(Z) \leq |v(Z) - v(Z_{j_{k_l}} )| + |v(Z_{j_{k_l}} ) - v_{j_{k_l}} (Z_{j_{k_l}} )| + |v_{j_{k_l}} (Z_{j_{k_l}} )| \leq 2\epsilon + \sup_{\pOj}|f|\le 3\epsilon+ \sup_{\pOinf} |f|,
\end{equation}
Letting $\epsilon\to 0$ we  get $0\le \sup_{\Oinf}v \leq \sup_{\pOinf} |f|$.

Let us recall that $\Omega_R=\Omega_\infty\cap B(0,2R)\subset \Omega_\infty$.  Since $v\in C(\R^n)$ with 
$v|_{\partial\Omega_\infty}=f$, and since $\spt f\subset B(0,R_0)$, for every $R>4R_0$ we have that $f_R|_{\partial\Omega_\infty}=f \eta(\cdot/R)\le v|_{\partial\Omega_\infty}$. Hence the maximum principle implies that $v_R \leq v$ in $\Omega_R$, and taking limits  we conclude that $ v_\infty \leq v$ on $\Oinf$. Write $0\le \widetilde{v}=v- v_\infty\in C(\overline{\Omega_\infty})$ so that $L_{\infty} \widetilde{v}=0$ in $\Oinf$ and $\widetilde{v}|_{\partial\Omega_\infty}=0$. For any $Z\in \Oinf$, since $\Oinf$ is a uniform domain with Ahlfors regular boundary, by Lemma \ref{lem:vanishing}
	  for any $\delta_\infty(Z)<R' < \diam(\pOinf)=\infty$ (see \ref{3-thm:pseudo-blow-geo} in Theorem \ref{thm:pseudo-blow-geo}) 
\begin{equation}\label{eqn:Holderinfty}
		0\le \widetilde{v}(Z) \lesssim \left( \frac{\delta_{\infty}(Z)}{R'} \right)^{\beta} \sup_{\Oinf} \widetilde{v} \leq 2 \left( \frac{\delta_{\infty}(Z)}{R'} \right)^{\beta} \sup_{\pOinf} f,    	\end{equation}
Letting $R'\to\infty$	we conclude that $\widetilde{v}(Z) = 0$ and hence $v=v_\infty$. Therefore, we have  shown that given a subsequence $\{v_{j_k}\}_k$ there is a further subsequence $\{v_{j_{k_l}}\}_l$ so that $v_{j_{k_l}}\to v_\infty$ uniformly on compact sets of $\RR^n$ and $\nabla v_{j_{k_l}} \rightharpoonup \nabla v_\infty$ in $L^2(\RR^n)$ as $l\to\infty$. This eventually shows that entire sequence $\{v_j\}$ satisfies $v_j\to v_\infty$ uniformly on compact sets of $\RR^n$ and $\nabla v_{j} \rightharpoonup \nabla v_\infty$ in $L^2(\RR^n)$ as $j\to\infty$.

Hence, in both \ref{CaseI} and \ref{CaseII}, if $0\le f\in \Lip(\partial\Omega_\infty)$ has compact support  then 
\begin{equation}\label{eqn:110A}
\lim_{j\to\infty} \int_{\pOj} f(q) d\oj^Z(q) =\lim_{j\to\infty} v_j(Z) = v_\infty(Z)=\int_{\pOinf} f(q) d\omega_{L_\infty}^Z(q),
\end{equation}
for any $Z\in\Omega_\infty$. A standard approximation argument and splitting each function on its positive and negative parts lead to shows that \eqref{eqn:110A} holds for all $f\in C_c(\RR^n)$, hence 
$\omega_j^Z \rightharpoonup \omega_{L_\infty}^Z$ as Radon measures for any $Z\in\Oinf$.

\medskip

Our next goal is to see that $\omega_{L_\infty}\in A_{\infty}(\sigma_{\infty})$ (where $\sigma_{\infty} = \mathcal{H}^{n-1}|_{\pOinf}$). Fix $p\in \pOinf$ and $0<r<\diam(\pOinf)$. Recall that whether $\diam(\pOinf)$ is finite or infinite, we always have $r< \diam(\Oj)$ for all $j$ sufficiently large. Let $\Delta' = B(m,s)\cap \pOinf$ with $m\in\pOinf$ and $B(m,s)\subset B(p,r) \cap \pOinf$. 
Let $A(p,r)\in \Omega_\infty$ be a corkscrew point relative to $\Delta(p,r)$ (whose existence is guaranteed by \ref{4-thm:pseudo-blow-geo} in Theorem \ref{thm:pseudo-blow-geo}). We can then find $p_j \in \pOj$ such that $p_j \to p$. Thus, for all $j$ large enough $B(p,r)\subset B(p_j,2r)$ and $\delta_j(A(p,r))\ge r/(2M)$. Hence, 
$A(p,r) $ is also a corkscrew point relative to $B(p_j,2r)\cap\partial\Omega_j$ in $\Oj$ with constant $4M$.  Since $m\in\pOinf$, we can also find $m_j\in\pOj$ such that $m_j\to m$. In particular, for $j$ sufficiently large
	\begin{equation}\label{eq:mjm}
		|m_j - m| < \frac{s}{5}.
	\end{equation}
	Note also that since all the $\Oj$'s are uniform and satisfy the CDC with the same constants, and all the operators $L_j$'s have ellipticity constants bounded below and above by $\lambda=1$ and $\Lambda$, we can conclude from Remark \ref{rem:doubling:needed} that there is a uniform constant $C_2$ depending on 
	$M, C_1, C_{AR}>1$, and
$\Lambda$, 
such that \eqref{doubling:needed} holds for all $\omega_j$ with the appropriate changes. Using this and  \cite[Theorem 1.24]{Ma} we obtain 
		\begin{multline}\label{eqn:113A}
			\oinf^{A(p,r)} (\Delta(m,s)) \geq \oinf^{A(p,r)} \left( \overline{B\left(m,\frac{4}{5}s\right)} \right)  \geq \limsup_{j\to\infty} \oj^{A(p,r)} \left( \overline{B\left(m,\frac{4}{5}s\right)} \right) 
			\\
			\geq \limsup_{j\to\infty} \oj^{A(p,r)} \left( \overline{B\left(m_j,\frac{3}{5}s\right)} \right) 
			\geq C_2^{-1} \limsup_{j\to\infty} \oj^{A(p,r)} \left(B\left(m_j,\frac{6}{5}s\right)\right),
		\end{multline}
		where we have used that $\delta_j(A(p,r))\ge r/(2M)\ge \frac35s/(2M)$.

Let $V$ be an arbitrary open set in $B(m,s)$, and note that by \eqref{eq:mjm}
$$
			V\subset B(m,s) \subset B\left(m_j, \frac{6}{5}s \right).
$$
	 Using again \cite[Theorem 1.24]{Ma}, we see that \eqref{eqn:113A} yields
		\begin{multline}
			\frac{\oinf^{A(p,r)}(V)}{\oinf^{A(p,r)}(\Delta(m,s))} 
			\leq 
			C_2 \dfrac{\liminf_{j\to\infty} \ojA(V)}{\limsup_{j\to\infty} \ojA\left(B\left(m_j,\frac{6}{5}s\right)\right)} 
			\\
			  \leq C_2 \liminf_{j\to\infty} \left( \dfrac{\ojA(V) }{\ojA\left(B\left(m_j,\frac{6}{5}s\right)\right) }\right). 
			\label{eqn:114A}
		\end{multline}
		The assumption $B(m,s) \subset B(p,r)$ implies $|m-p|\leq r-s$. Using this and that $m_j\to m$, $p_j\to p$ as $j\to\infty$ one can easily see that 
		$|m_j - p_j|<r-\frac{s}{5}$ for all $j$ large enough and hence
		\begin{equation}\label{tempincl}
			B\left(m_j, \frac{6}{5}s \right)\cap\partial\Omega_j \subset B(p_j,2r)\cap\partial\Omega_j.
		\end{equation}
		As mentioned above $A(p,r)$ is a corkscrew point relative to $B(p_j,2r)\cap\partial\Omega_j$ in $\Oj$. This, \eqref{tempincl} and the fact that by assumption, $\omega_j\in A_\infty(\sigma_j)$ with uniform constants $C_0, \theta$ allow us to conclude that
		\begin{equation}\label{eq:Ainftyoj}
			\dfrac{\ojA(V) }{\ojA\left(B\left(m_j,\frac{6}{5}s\right)\right) } \leq C_0 \left( \dfrac{\sj(V)}{\sj\left(B\left(m_j, \frac{6}{5}s \right) \right)} \right)^\theta
			\le C_0 C_{AR}^\theta \left( \dfrac{\sj(V)}{s^{n-1}} \right)^\theta,
		\end{equation}
		where in the last estimate we have used that $\partial\Omega_j$ is Ahlfors regular with constants $C_{AR}$.
		Combining \eqref{eqn:114A}, \eqref{eq:Ainftyoj}, the fact that $\sj \rightharpoonup \sinf$, \cite[Theorem 1.24]{Ma}, and \ref{5-thm:pseudo-blow-geo} in Theorem \ref{thm:pseudo-blow-geo}, we finally arrive at 
		\begin{multline*}
			\frac{\oinf^{A(p,r)}(V)}{\oinf^{A(p,r)}(\Delta(m,s))} 
			\leq C_0 C_{AR}^\theta \left(\liminf_{j\to\infty} \dfrac{\sj(V)}{s^{n-1}} \right)^\theta 
			\\
			\leq C_0 C_{AR}^\theta \left( \frac{\sinf(\overline V)}{s^{n-1}} \right)^\theta 
			\le 
			C_0 C_{AR}^{4\theta}  2^{8(n-1)\theta} \left( \frac{\sigma_\infty(\overline V)}{\sigma_\infty(\Delta(m,s))} \right)^\theta.
		\end{multline*}
and therefore we have shown that  for any open set $V\subset B(m,s)$ there holds 
\begin{equation}\label{eq:Ainftyoinf}
			\frac{\oinf^{A(p,r)}(V)}{\oinf^{A(p,r)}(\Delta(m,s))} 
			\le 
			C_0 C_{AR}^{4\theta}  2^{8(n-1)\theta} \left( \frac{\sigma_\infty(\overline V)}{\sigma_\infty(\Delta(m,s))} \right)^\theta.			
\end{equation}

Consider next an arbitrary Borel set $E\subset B(m,s)$. Since $\sigma_{\infty}$ and  $\oinf^{A(p,r)}$ are
		Borel regular, given any $\epsilon>0$ there is an open set $U$ and a compact set $F$ so that  $F\subset E\subset U \subset B(m,s)$ and $\oinf^{A(p,r)}(U\setminus F)+ \sigma_{\infty}(U\setminus F) <\epsilon$. Note that for any $x\in F$, there is $r_x >0 $ such that $B(x,2r_x) \subset U$. Using that $F$ is compact we can then show there exists a finite collection of points $\{x_i\}_{i=1}^m\subset F$ such that $F\subset \bigcup_{i=1}^m B(x_i, r_i) =:  V$ and $B(x_i, 2r_i) \subset U$ for $i= 1,\dots,m$. Consequently, $F\subset V \subset \overline V \subset U$ and $\sigma_{\infty}(\overline V \setminus F) \leq \sigma_{\infty}(U\setminus F) <\epsilon$. We next use \eqref{eq:Ainftyoinf} with $V$ to see that
\begin{multline*}
\frac{\oinf^{A(p,r)}(E)}{\oinf^{A(p,r)}(\Delta(m,s))} 
\leq 
\frac{\epsilon +\oinf^{A(p,r)}(F)}{\oinf^{A(p,r)}(\Delta(m,s))} 
\leq 
\frac{\epsilon +\oinf^{A(p,r)}(V)}{\oinf^{A(p,r)}(\Delta(m,s))} 
\\
\leq 
\frac{\epsilon}{\oinf^{A(p,r)}(\Delta(m,s))} 
+
C_0 C_{AR}^{4\theta}  2^{8(n-1)\theta} \left( \frac{\sigma_{\infty}(\overline V)}{\sigma_{\infty}(\Delta(m,s))} \right)^\theta 
\\
\leq 
\frac{\epsilon}{\oinf^{A(p,r)}(\Delta(m,s))} 
+
C_0 C_{AR}^{4\theta}  2^{8(n-1)\theta} \left( \frac{\sigma_{\infty}(E)+\epsilon}{\sigma_{\infty}(\Delta(m,s))} \right)^\theta. \label{eqn:119A}
\end{multline*}
Letting $\epsilon \to 0$ we obtain  as desired that $\oinf\in A_\infty(\sigma_\infty)$ with constants $C_0 C_{AR}^{4\theta}  2^{8(n-1)\theta} $ and $\theta$ and the proof is  complete. 
\end{proof}

\section{Proof of the Main Theorem}

Applying Theorem \ref{thm:pseudo-blow-geo}, we obtain that $\Omega_\infty$ is a uniform domain with constants $4M$ and $2C_1$, whose boundary is Ahlfors regular with constant $2^{5(n-1)}C_{AR}^2$. Moreover, Theorem \ref{thm:blow-ana-pole} gives that $\omega_{L_{\infty}}\in A_\infty(\sigma_\infty)$ with constants 
$\widetilde{C}_0=C_2 C_{AR}^{4\theta} 2^{8(n-1)\theta}$ and $\widetilde{\theta}=\theta$. 
Here $L_\infty=-\divg(\mathcal{A}^* \nabla)$ with $\mathcal{A}^*$  a constant-coefficient real
 symmetric uniformly elliptic matrix with ellipticity constants 
$1=\lambda\le \Lambda<\infty$.
We can then invoke Theorem \ref{thm:hmu}, 
to see that $\Oinf$ satisfies the exterior corkscrew condition with 
constant
\[N_0=N_0(4M,2C_1, 2^{5(n-1)}C_{AR}^2,\Lambda, C_0  C_2 C_{AR}^{4\theta} 2^{8(n-1)\theta},\theta)\] 
(see introduction to Section \ref{comp-tt}). Therefore,  since $0\in\pOinf$,  $0<\frac12<\diam(\pOinf)$ (recall that $\diam(\pOinf) = \infty$ in \ref{CaseI}, and $\diam(\pOinf) = \diam(\Oinf) = R_0 \geq 1$) there exists $A_0=A^-(0,\frac12)$ so that 
\begin{equation}\label{eqn:201A}
	B\left( A_0,\frac{1}{2N_0} \right) \subset B\left(0,\frac12\right)\setminus\overline{\Oinf}.
\end{equation}
Hence
\begin{equation}
	\dist\left(B\left(A_0,\frac{1}{4N_0}\right), \R^n\setminus \overline{\Oinf}\right) 
	\geq \frac{1}{4N_0}.
\end{equation} 
Since $\overline{\Oj} \to \overline{\Oinf}$, it follows that for all $j$ large enough
\begin{equation}\label{eqn:203A}
	B\left( A_0,\frac{1}{4N_0} \right) 
	\subset  B\left(0,\frac12\right)\setminus\overline{\Oj} \subset B(0,1)\setminus\overline{\Oj}.
\end{equation}
Hence for all $j$ large enough $A_0$ is a corkscrew point relative to $B(0,1)\cap\partial\Omega_j$ for $\R^n\setminus \overline{\Oj}$ with constant $4N_0$. This contradicts our assumption that  $\Omega_j$ has no exterior corkscrew point with constant $N=4N_0$ for the surface ball $B(0,1)\cap\partial\Omega_j$ and the proof is complete.

\end{document}